\documentclass[10pt,reqno]{amsart}
\usepackage{amssymb,mathrsfs,color}
\allowdisplaybreaks
\usepackage{enumerate}
\usepackage{graphicx}
\usepackage{xcolor} 
\usepackage{geometry}
\usepackage{amsfonts,amssymb,amsmath,mathtools}
\usepackage{slashed}
\usepackage{ mathrsfs }
\usepackage[titletoc,title]{appendix}
\usepackage{wrapfig}

\usepackage{cancel}

\usepackage{cite}
\usepackage{framed}
\definecolor{green}{rgb}{0,0.8,0} 
\makeatletter
\renewcommand{\maketag@@@}[1]{\hbox{\m@th\normalsize\normalfont#1}}
\makeatother
\title[A Radiation Hydrodynamics Model with Viscosity and Thermal Conductivity]{Global Regularity for a Radiation Hydrodynamics Model with Viscosity and Thermal Conductivity}

\author[J.-H. Zhang]{Junhao Zhang}
\address[Junhao Zhang]{\newline School of Mathematics and Statistics, Wuhan University, Wuhan 430072, China\\
 \newline Computational Science Hubei Key Laboratory, Wuhan University, Wuhan 430072, China}
\email{zhangjunhao@whu.edu.cn}

\author[H.-J. Zhao]{Huijiang Zhao}
\address[Huijiang Zhao]{\newline School of Mathematics and Statistics, Wuhan University, Wuhan 430072, China\\
 \newline Computational Science Hubei Key Laboratory, Wuhan University, Wuhan 430072, China}
\email{hhjjzhao@whu.edu.cn}

\newcommand{\sgn}{{\mathrm{sgn}}}

\definecolor{deepgreen}{cmyk}{1,0,1,0.5}

\makeatletter

\newcommand{\Rmnum}[1]{\expandafter\@slowromancap\romannumeral #1@}
\makeatother

\newcommand{\Del}[1]{}

\numberwithin{equation}{section}

\newtheorem{theorem}{Theorem}[section]
\newtheorem{lemma}[theorem]{Lemma}
\newtheorem{remark}[theorem]{Remark}
\newtheorem{definition}[theorem]{Definition}

\begin{document}
\date{\today}



\begin{abstract}
In this paper, we study the global wellposedness of a radiation hydrodynamics model with viscosity and thermal conductivity. It is now well-understood that, unlike the compressible Euler equations whose smooth solutions must blow up in finite time no matter how small and how smooth the initial data is, the dissipative structure of such a radiation hydrodynamics model can indeed guarantee that its one-dimensional Cauchy problem admits a unique global smooth solution provided that the initial data is sufficiently small, while for large initial data, even if the heat conductivity is taken into account but the viscosity effect is ignored, shock type singularities must appear in finite time for smooth solutions of the Cauchy problem of one-dimensional radiation hydrodynamics model with thermal conductivity and zero viscosity. Thus a natural question is, if effects of both the viscosity and the thermal conductivity are considered, does the one-dimensional radiation hydrodynamics model with viscosity and thermal conductivity exist a unique global large solution? We give affirmative answer to this problem and show in this paper that the initial-boundary value problem to the radiation hydrodynamics model in an one-dimensional periodic box $\mathbb{T}\cong \mathbb{R}/\mathbb{Z}$ with viscosity and thermal conductivity does exist a unique global smooth solution for any large initial data. The main ingredient in our analysis is to introduce some delicate estimates, especially  an improved $L^m([0,T], L^\infty(\mathbb{T}))-$estimate on the absolute temperature for some $m\in\mathbb{N}$ and a pointwise estimate between the absolute temperature, the specific volume, and the first-order spatial derivative of the macro radiation flux, to deduce the desired positive lower and upper bounds on the density and the absolute temperature.
\\[2mm]
\\[2mm]
\end{abstract}

\maketitle \centerline{\date}

\tableofcontents


\section{Introduction and main result}

The dynamics of astrophysical flows, where a gas interacts with radiation through energy exchanges, can be described by the following system of equations, cf. \cite{Castor-2004, Chandrasekhar-1960, Mihalas-Mihalas-1984, Pomraning-2005} and the references cited therein:
\begin{eqnarray}\label{MD-Radiation-hydrodynamics}
\rho_{t}+\operatorname{div}_{\bf y}(\rho {\bf u}) &=&0, \nonumber\\
(\rho {\bf u})_{t}+\operatorname{div}_{\bf y}(\rho {\bf u} \otimes {\bf u})+\nabla_{\bf y} p &=&0, \nonumber\\
(\rho \mathscr{E})_{t}+\operatorname{div}_{\bf y}(\rho \mathscr{E} {\bf u}+p {\bf u})+\operatorname{div}_{\bf y} {\bf q} &=&0,\\
-\nabla_{\bf y} \operatorname{div}_{\bf y} {\bf q}+a {\bf q}+b \nabla_{\bf y}\left( \theta^{4}\right) &=&0.\nonumber
\end{eqnarray}
Here ${\bf y}=(y_1, y_2, y_3)\in\Omega\subseteq\mathbb{R}^3$ is the Eulerian space variable with $\Omega$ being some open set in $\mathbb{R}^3$, $t\in\mathbb{R}^+$ is the time variable, the positive constants $a$ and $b$ stand for the absorption coefficient and the Stefan-Boltzmann constant respectively, and the primary dependent variables are the fluid density $\rho$, its bulk velocity ${\bf u}=(u_1,u_2, u_3)\in\mathbb{R}^3$, its absolute temperature $\theta$, and the radiative heat flux ${\bf q}=(q_1,q_2,q_3)\in\mathbb{R}^3$. The specific total energy $\mathscr{E}=e+\frac{1}{2}|{\bf u}|^{2}$, where $e$ is the specific internal energy. The pressure $p$ and the internal energy $e$ are prescribed through constitutive relations as functions of $\rho$ and/or $\theta$. The thermodynamic variables $\rho, p, e$, and $\theta$ are related through Gibbs' equation $de =\theta dS-pd\rho^{-1}$ with $S$ being the specific entropy. Throughout this paper, we consider only ideal, polytropic gases:
\begin{equation}\label{the state equation}
  p=R\rho \theta=A\rho^\gamma\exp\left(\frac{\gamma-1}{R}S\right),\quad e=c_v\theta:=\frac{R\theta}{\gamma-1},
\end{equation}
where the positive constants $A, R$ are the specific gas constants, $c_v$ is the specific heat at constant volume, respectively, and $\gamma>1$ is the adiabatic exponent.

System \eqref{MD-Radiation-hydrodynamics} can also be used to describe similar questions arising in the modeling of reentry problems, or high temperature combustion phenomena and can be formally derived by asymptotic arguments, starting from a more complete physical system consisting of a kinetic equation for the specific intensity of radiation coupled with the Euler system describing the evolution of the fluid; see the appendix of \cite{Lin-Coulombel-Goudon-PhysicaD-2006} and the
references therein. We also refer to \cite{Castor-2004, Chandrasekhar-1960, Mihalas-Mihalas-1984, Pomraning-2005, Zeldovich-Raizer-Book-2002} for the physical background.

Although the theory of symmetric positive partial differential equations developed by K.O. Friedrichs, C.-H. Gu et al in \cite{Friedrichs-CPAM-1954, Friedrichs-CPAM-1958, Gu-ActaMathSinica-1964} can be applied to the system \eqref{MD-Radiation-hydrodynamics} to yield a satisfactory local wellposedness theory to its Cauchy problem or its initial-boundary value problem with suitable boundary conditions, the corresponding global wellposedness theory is established only for the one-dimensional case and, to the best of our knowledge, no result has been obtained for the multidimensional case up to now.

In fact, for the Cauchy problem of the one-dimensional radiation hydrodynamics system
\begin{eqnarray}\label{1D-Radiative-Hydrodynamics-Euler}
\rho_t+(\rho u)_y&=&0, \nonumber\\
(\rho u)_t+\left(\rho u^{2}+p\right)_y&=&0,\\
(\rho \mathscr{E})_t+(\rho \mathscr{E} u+p u)_y&=&-q_y,\nonumber\\
-q_{yy}+aq+b\left(\theta^{4}\right)_y&=&0\nonumber
\end{eqnarray}
with prescribed initial data
\begin{equation}\label{1D-Radiation-Hydrodynamics-Euler-IC}
(\rho(0,y), u(0,y), \theta(0,y))=(\rho_0(y), u_0(y),\theta_0(y)),\quad y\in\mathbb{R},
\end{equation}
it is shown in \cite{Deng-Yang-M2AS-2020, Fan-Ruan-Xiang-AIHPAN-2019, Lin-CMS-2011, Lin-Coulombel-Goudon-CRMathAcadSciParis-2007, Lin-Coulombel-Goudon-PhysicaD-2006, Lin-Goudon-CAM-2011, Rohde-Wang-Xie-CPAA-2013, Wang-Xie-JDE-2011, Xie-DCDS-B-2012} that, unlike the Cauchy problem of the one-dimensional compressible Euler system whose smooth solutions must blow up in finite time no matter how small and how smooth the initial data is, the dissipative structure of the system \eqref{1D-Radiative-Hydrodynamics-Euler} can indeed guarantee that the Cauchy problem \eqref{1D-Radiative-Hydrodynamics-Euler}, \eqref{1D-Radiation-Hydrodynamics-Euler-IC} does admit a unique global smooth solution $(\rho(t,y), u(t,y), \theta(t,y), q(t,y))$ provided that the initial data $(\rho_0(y), u_0(y), \theta_0(y))$ is assumed to be sufficiently smooth and suitably small. For similar results on its initial-boundary value problem on the half line $\mathbb{R}^+$, those interested are referred to \cite{Fan-Ruan-Xiang-SIMA-2019, Fan-Ruan-Xiang-DCDS-2021} on its inflow problem.

The problem on the global smooth solvability of the Cauchy problem \eqref{1D-Radiative-Hydrodynamics-Euler}, \eqref{1D-Radiation-Hydrodynamics-Euler-IC} with large initial data is quite difficult and subtle. For result in this direction, S.-X. Li and J. Wang studied in \cite{Li-Wang-2021} the following Cauchy problem
\begin{eqnarray}\label{1D-Radiative-Hydrodynamics-HC-ZeroV}
\rho_t+(\rho u)_y&=&0, \nonumber\\
(\rho u)_t+\left(\rho u^{2}+p\right)_y&=&0,\\
(\rho \mathscr{E})_t+(\rho \mathscr{E} u+p u)_y+q_y&=&\left(\kappa(\rho)\theta_y\right)_y,\nonumber\\
-q_{yy}+aq+b\left(\theta^{4}\right)_y&=&0\nonumber
\end{eqnarray}
with prescribed initial data
\begin{equation}\label{1D-Radiation-Hydrodynamics-HC-ZeroV-IC}
(\rho(0,y), u(0,y), \theta(0,y))=(\rho_0(y), u_0(y),\theta_0(y)),\quad y\in\mathbb{R}.
\end{equation}
The result obtained in \cite{Li-Wang-2021} shows that, even if the heat conduction effect is taken into account, i.e. $\kappa(\rho)>0$ for $\rho>0$, while the viscosity is ignored, shock type singularities must appear in finite time for solutions of the Cauchy problem \eqref{1D-Radiative-Hydrodynamics-HC-ZeroV}, \eqref{1D-Radiation-Hydrodynamics-HC-ZeroV-IC} if the initial data is suitably large in the sense that the first-order derivatives of the initial data are less than some negative constant. Thus a natural question is: If the effects of both  the viscosity and thermal conductivity are considered, does the one-dimensional radiation hydrodynamics system with viscosity and thermal conductivity admit a unique global smooth nonvacuum solution for any large initial data? The main purpose of this paper is to give a positive answer to such a problem.

To this end, we study the initial-boundary value problem of the radiation hydrodynamics system with both viscosity and thermal conductivity in an one-dimensional periodic box $\mathbb{T}\cong \mathbb{R}/\mathbb{Z}=\left[-\frac 12,\frac 12\right]$. Such an initial-boundary value problem can be rewritten in the Lagrangian coordinates as
\begin{eqnarray}\label{Radiation-NS}
v_{t}-u_{x}&=&0,\nonumber\\
u_{t}+p_{x}&=&\left(\frac{\mu u_{x}}{v}\right)_{x},\\
\left(e+\frac{u^{2}}{2}\right)_{t}+(pu)_{x}+q_x&=&\left(\frac{\mu u u_{x}}{v}\right)_{x} +\left(\frac{\kappa(v,\theta)\theta_x}{v}\right)_{x},\nonumber\\
-\left(\frac{q_{x}}{v}\right)_x+avq+b(\theta^4)_x&=&0\nonumber
\end{eqnarray}
with prescribed initial data and periodic boundary condition
\begin{eqnarray}\label{Radiation-NS-IC}
(v(0,x),u (0,x), \theta(0,x))&=&\left(v_0(x), u_0(x), \theta_0(x)\right),\\
(v(t,x+1), u(t,x+1), \theta(t,x+1), q(t,x+1))&=& (v(t,x), u(t,x), \theta(t,x), q(t,x)),\nonumber
\end{eqnarray}
where $t\in\mathbb{R}$ is the time variable, $x\in \mathbb{T}\cong \mathbb{R}/\mathbb{Z}=\left[-\frac 12,\frac 12\right]$ is the Lanrangian space variable, and $v:=\rho^{-1}$ is the specific volume. The pressure $p$ and the specific internal energy $e$ satisfy the constitutive relations \eqref{the state equation} for ideal, polytropic gases.

For the transport coefficients $\mu\geq 0$ (viscosity) and $\kappa\geq 0$ (heat conductivity), since the radiation process
involves high temperature and experimental results for gases at high temperatures in \cite{Zeldovich-Raizer-Book-2002} show that both $\mu$ and $\kappa$ may depend on the specific volume $v$ and the absolute temperature $\theta$, we thus assume, as in \cite{Jiang-ZHeng-JMP-2012, Jiang-ZHeng-ZAMP-2014, Kawohl-JDE-1985, Liao-Zhao-JDE-2018, Liao-Zhao-Zhou-SIAM-2021, Qin-Hu-Wang-QuartApplMath-2011, Qin-Zhang-Su-Cao-JMFM-2016, Umehara-Tani-JDE-2007, Umehara-Tani-ProcJapanAcad-2008}, that the thermal conductivity $\kappa=\kappa\left(v,\theta\right)$ takes the form
\begin{eqnarray}\label{Kappa}
\kappa\left(v,\theta\right)=\kappa_{1}+\kappa_{2}v\theta^{\beta}
\end{eqnarray}
with $\kappa_{1}$, $\kappa_{2}$, and $\beta$ being some positive constants. As for the viscosity, since as pointed out in \cite{Jenssen-Karper-SIMA-2010} that even for the problem on the construction of global large smooth solutions to one-dimensional compressible Navier-Stokes equations for a viscous and heat conducting ideal polytropic gas, the case of ``{\it temperature dependence of the viscosity $\mu$ has turned out to be especially problematic}," we thus first consider the case when the viscosity $\mu$ is a positive constant. It is worth to pointing out that similar result also holds for the case when $\mu$ is a smooth function of the specific volume $v$.

With the above preparations in hand, we now turn to state our main result.
\begin{theorem}\label{main theorem}
Assume that $\kappa(v,\theta)$ satisfies \eqref{Kappa} and let $\beta>13$. If the initial data $(v_{0}(x), u_{0}(x), \theta_{0}(x),$ $q_0(x))$ satisfies the compatibility condition
$$
-\frac{d}{dx}\left(\frac{1}{v_0(x)}\frac{dq_{0}(x)}{dx}\right)+av_0(x)q_0(x)+b\frac{d\left(\left|\theta_0(x)\right|^4\right)}{dx}=0
$$
and
\begin{eqnarray*}
\left(v_{0}(x), u_{0}(x), \theta_{0}(x)\right) &\in& H^{2}(\mathbb{T}),\\
\inf _{x \in \mathbb{T}}\Big\{v_0(x), \theta_{0}(x)\Big\}&>&0,
\end{eqnarray*}
then the periodic initial-boundary value problem \eqref{Radiation-NS}, \eqref{Radiation-NS-IC} possesses a unique global smooth solution $(v(t, x),$ $u(t, x), \theta(t, x), q(t, x))$, which satisfies for any given positive constant $T>0$ that
\begin{eqnarray}\label{Esitmiates-on-solutions}
\underline{V}(T) \leq v(t,x) &\leq& \overline{V}(T),\quad \forall(t, x) \in\left[0, T\right] \times \mathbb{T},\nonumber\\
\underline{\Theta}(T) \leq \theta(t,x) &\leq& \overline{\Theta}(T),\quad \forall(t, x) \in\left[0, T\right] \times \mathbb{T},
\end{eqnarray}
and
\begin{equation}
  \|(v(t), u(t), \theta(t), q(t))\|_{H^{2}(\mathbb{T})}\leq C(T), \quad t\in[0,T],
\end{equation}
where $\underline{V}(T)$, $\overline{V}(T)$, $\underline{\Theta}(T)$, $\overline{\Theta}(T)$, and $C(T)$ are some positive constants depending only on $\left\|\left(v_{0}, u_{0}, \theta_{0}\right)\right\|_{H^{2}(\mathbb{T})}$, $\inf _{x \in \mathbb{T}}\{v_0(x), \theta_{0}(x)\}$, and $T$.
\end{theorem}
\begin{remark} Several remarks concerning Theorem \ref{main theorem} are listed below:
\begin{itemize}
\item Without loss of generality, we assume that $\int_{\mathbb{T}}v_0(x)dx=1$ in the rest of this paper. Under such an assumption, the map $\Psi$ defined by \eqref{Euler to Lagrange} maps the region $\mathbb{R}^+\times \mathbb{T}$ in the Lagrangian coordinates into the region $\mathbb{R}^+\times \mathbb{T}$ in the Eulerian coordinates.
\item It would be an interesting problem to study the large behavior of the global solutions  $(v(t,x), u(t,x),$ $\theta(t,x), q(t,x))$ constructed in Theorem \ref{main theorem}. For this purpose, we need to improve the time-dependent estimates on $(v(t,x), u(t,x), \theta(t,x), q(t,x))$ obtained in Theorem \ref{main theorem} such that  certain time-independent estimates, which are sufficient to yield their time-asymptotic behavior, can be obtained.

\item It seems that the arguments used in this paper work only for the initial-boundary value problem \eqref{Radiation-NS}, \eqref{Radiation-NS-IC} in the one-dimensional periodic box $\mathbb{T}\cong \mathbb{R}/\mathbb{Z}$. Thus the problem on the construction of global large smooth solutions to its Cauchy problem and its initial-boundary value problems in the bounded open interval $(0,1)$ or in the half line $(0,+\infty)$ together with the precise description of their large time behavior is still unknown.

\item As mentioned above, the radiation hydrodynamics model \eqref{MD-Radiation-hydrodynamics} can be formally derived from a more complete physical system consisting of a kinetic equation for the specific intensity of radiation coupled with the Euler system describing the evolution of the fluid, it  would be an interesting problem to provide a rigorous proof to such a formal derivation. For some discussions on this problem, those interested are referred to \cite{Buet-Despres-JQuantSpectroscRadiatTransfer-2004, Lafitte-Godillon-Goudon-MultiscaleModelSimul-2005, Lowrie-Morel-Hittinger-AstrophysicalJ-1999} and the references cited therein.
\end{itemize}
\end{remark}

Now let us list the main difficulties encountered and our main ideas to resolve them. With the aim of constructing global smooth large solutions to the periodic initial-boundary value problem \eqref{Radiation-NS}, \eqref{Radiation-NS-IC}, our approach is based on the continuation method, which is a combination of local existence result with certain {\it a priori} energy type estimates. As an one-dimensional Navier-Stokes type equations, to guarantee the global solvability of the periodic initial-boundary value problem \eqref{Radiation-NS}, \eqref{Radiation-NS-IC} for large initial data, the main task is to deduce positive lower and upper bounds on the specific volume $v$ and the absolute temperature $\theta$.

For one-dimensional Navier-Stokes type equations, some effective methods, which rely heavily on the structure of the systems under consideration, or more precisely on the assumptions imposed on the transport coefficients and on the constitutive relations satisfied by the five thermodynamic variables $v$, $p$, $e$, $\theta,$ and $S$, have been developed to yield the desired positive lower and upper bounds on the specific volume and the absolute temperature and consequently have been successfully applied to construct global smooth large solutions to viscous and heat conducting ideal polytropic gases \cite{Antontsev-Kazhikov-Monakhov-1990, Huang-Wang-IUMJ-2016, Jenssen-Karper-SIMA-2010, Jiang-AMPA-1998, Jiang-CMP-1999, Jiang-PRSE-2002, Kanel-DU-1968, Kazhikhov-Shelukhin-JAMM-1977, Li-Liang-ARMA-2016, Liu-Yang-Zhao-Zou-SIMA-2014, Tan-Yang-Zhao-Zou-SIMA-2013}, to the real gas \cite{Kawohl-JDE-1985}, to nonlinear thermoviscoelasticity \cite{Dafermos-SIAM-1982, Dafermos-Hsiao-NonA-1982}, and to a viscous radiative and reactive gas \cite{Gong-He-Liao-SCM-2021, Gong-Liao-Xu-Zhao-2022, Gong-Xu-Zhao-MAA-2022, He-Liao-Wang-Zhao-ActaMathSci-2018, Jiang-ZHeng-JMP-2012, Jiang-ZHeng-ZAMP-2014, Liao-Wang-Zhao-JDE-2019, Liao-Zhao-JDE-2018, Qin-Hu-Wang-QuartApplMath-2011, Qin-Zhang-Su-Cao-JMFM-2016, Umehara-Tani-JDE-2007, Umehara-Tani-ProcJapanAcad-2008}.

To illustrate the main difficulties we encounter in this paper, we show in the following the fundamental estimates used in \cite{Antontsev-Kazhikov-Monakhov-1990, Huang-Wang-IUMJ-2016, Jenssen-Karper-SIMA-2010, Jiang-AMPA-1998, Jiang-CMP-1999, Jiang-PRSE-2002, Kanel-DU-1968, Kazhikhov-Shelukhin-JAMM-1977, Li-Liang-ARMA-2016, Liu-Yang-Zhao-Zou-SIMA-2014, Tan-Yang-Zhao-Zou-SIMA-2013} to construct global smooth nonvacuum large solutions to the Cauchy problem of the one-dimensional viscous and heat conducting ideal polytropic gases with constant viscosity and density and temperature dependent heat conductivity around the constant equilibrium state $(v,u,\theta)=(1,0,1)$. For such a problem, if we denote the normalized entropy around the constant equilibrium state $(v,u,\theta)=(1,0,1)$ by $\eta(v,u,\theta):=\frac 12 u^2+R(v-\ln v-1)+c_v(\theta-\ln \theta-1)$, then the analysis in \cite{Antontsev-Kazhikov-Monakhov-1990, Huang-Wang-IUMJ-2016, Jenssen-Karper-SIMA-2010, Jiang-AMPA-1998, Jiang-CMP-1999, Jiang-PRSE-2002, Kanel-DU-1968, Kazhikhov-Shelukhin-JAMM-1977, Li-Liang-ARMA-2016, Liu-Yang-Zhao-Zou-SIMA-2014, Tan-Yang-Zhao-Zou-SIMA-2013} are based on the following two types of fundamental estimates:
\begin{itemize}
\item [(1).] For each given $R>0$ and each $t>0$, the $R-$independent estimate on $\|(v(t),\theta(t))\|_{L^1([-R,R])}$ and the dissipative estimates on the first-order spatial derivatives of the bulk velocity $u$ and the absolute temperature, that is $\int^t_0\int_{\mathbb{R}}\left(\frac{\mu u_x^2}{v\theta}+\frac{\kappa(v,\theta)\theta^2_x}{v\theta^2}\right)dxds$, which follow from the Jenssen inequality and the basic energy estimates obtained through the normalized entropy $\eta(v,u,\theta)$ which is based on the following identity for $\eta(v,u,\theta)$:
\begin{equation}\label{Identity-entropy-NS}
\eta(v,u,\theta)_t+\frac{\mu u_x^2}{v\theta}+\frac{\kappa(v,\theta)\theta^2_x}{v\theta^2}=\left[\frac{\mu uu_x}{v}
+\left(1-\frac 1\theta\right)\frac{\kappa(v,\theta)\theta_x}{v}+R\left(1-\frac\theta v\right)u\right]_x;
\end{equation}
\item [(2).] The local-in-time estimate on the positive lower bound on the absolute temperature $\theta$ in terms of the lower bound estimate on the specific volume $v$, which is based on the maximum principle and the following equation satisfied by $\frac 1\theta$:
\begin{equation}\label{Identity-temperature-NS}
c_v\left(\frac 1\theta\right)_t=\left[\frac{\kappa(v,\theta)}{v}\left(\frac 1\theta\right)_x\right]_x
-\frac{2\theta\kappa(v,\theta)}{v}\left[\left(\frac 1\theta\right)_x\right]^2
-\frac{\mu}{v\theta^2}\left[u_x-\frac{R\theta}{2\mu}\right]^2+\frac{R^2}{4\mu v}.
\end{equation}
\end{itemize}

We note, however, if the radiation effect is taken into account, then, on the one hand, from the corresponding identity satisfied by the normalized entropy $\eta(v,u,\theta)$, one can only deduce that the solution $(v(t,x), u(t,x), \theta(t,x), q(t,x))$ of the periodic initial-boundary value problem \eqref{Radiation-NS}, \eqref{Radiation-NS-IC} satisfies the following basic energy estimate obtained through the normalized entropy $\eta(v,u,\theta)$:
\begin{eqnarray}\label{entropy identity}
  &&\int_{\mathbb{T}}\eta(v,u,\theta)dx+\int_{0}^{t}\int_{\mathbb{T}}\left(\frac{\mu u_x^2}{v\theta}+\frac{\kappa(v,\theta)\theta_x^2}{v\theta^2}\right)dxds\nonumber\\
  &&+\int_{0}^{t}\int_{\mathbb{T}}\left(\frac{avq^2}{4b\theta^5}+\frac{q_x^2}{4bv\theta^5}\right)dxds\\
  &=&\int_{\mathbb{T}}\eta(v_0,u_0,\theta_0)dx+\int_{0}^{t}\int_{\mathbb{T}}\frac{5q\theta_xq_x}{4bv\theta^6}dxds.\nonumber
\end{eqnarray}

The last term in the right hand side of \eqref{entropy identity} is a cubic nonlinear term, to the best of our knowledge, there are two ways to control such a term suitably to yield the desired fundamental estimates mentioned above:
\begin{itemize}
\item If one focuses on the construction of small amplitude solutions to the periodic initial-boundary value problem \eqref{Radiation-NS}, \eqref{Radiation-NS-IC}, such a term can indeed be absorbed by the terms in the left hand side of \eqref{entropy identity};

\item For the case when the adiabatic exponent $\gamma>1$ satisfying $\gamma-1>0$ sufficiently small, from the constitutive equation \eqref{the state equation}, one can deduce that
\begin{equation}\label{Nishida-Smoller-1}
\theta=\frac AR v^{1-\gamma}\exp\left(\frac{\gamma-1}{R}S\right).
\end{equation}
Moreover, one can deduce from the equation \eqref{Radiation-NS}$_3$ for the conservation of energy that
the absolute temperature $\theta(t,x)$ satisfies the following second-order parabolic equation:
\begin{equation}\label{Equation-theta}
    \frac{R}{\gamma-1}\theta_t-\left(\frac{\kappa(v,\theta)\theta_x}{v}\right)_x=-q_x+\frac{\mu u_x^2}{v}-\frac{R\theta}{v}u_x=-q_x+\frac{\mu}{v}\left(u_x-\frac{R\theta}{2\mu}\right)^2-\frac{R^2\theta^2}{4\mu v}.
\end{equation}
It is easy to see from \eqref{Nishida-Smoller-1} and \eqref{Equation-theta} that if $\gamma-1$ is chosen sufficiently small, then both $\|\theta(t)-1\|_{L^\infty(\mathbb{T})}$ and $\|\theta_x(t)\|_{L^\infty(\mathbb{T})}$ can be chosen as small as wanted and thus the last term in the right hand side of \eqref{entropy identity} can also be absorbed by the terms in the left hand side of \eqref{entropy identity}. As a consequence, a Nishida-Smoller type global solvability result for the initial-boundary value problem \eqref{Radiation-NS}, \eqref{Radiation-NS-IC} for a class of large initial data can be obtained, for results in this direction, see \cite{Hong-NonliAnal-RWA-2017} for the one-dimensional Cauchy problem of \eqref{Radiation-NS} and \cite{Wan-Wu-M2AS-2019, Zhu-M2AS-2020} for radial symmetric solutions to the corresponding multidimensional radiation hydrodynamics system with viscosity and thermal conductivity in an exterior domain and a bounded concentric annular domain. For the corresponding original Nishida-Smoller type global existence result for one-dimensional ideal polytropic isentropic compressible Euler system, please refer to \cite{Nishida-Smoller-CPAM-1973}, whereas the nonisentropic case was analyzed by Liu \cite{Liu-IUMJ-1977}  and also Temple \cite{Temple-JDE-1981}; while for the corresponding Nishida-Smoller type result for one-dimensional compressible flow of a viscous and heat-conducting ideal fluid with density and temperature dependent transport coefficients, those interested are referred to \cite{Liu-Yang-Zhao-Zou-SIMA-2014}.
\end{itemize}

For the problem on the global solvability of the periodic initial-boundary value problem \eqref{Radiation-NS}, \eqref{Radiation-NS-IC} with large initial data for any adiabatic exponent $\gamma>1$, such a term can not be controlled suitably to yield the desired fundamental estimates on $\|(v(t),\theta(t))\|_{L^1(\mathbb{T})}$ and $\int^t_0\int_{\mathbb{T}}\left(\frac{\mu u_x^2}{v\theta}+\frac{\kappa(v,\theta)\theta^2_x}{v\theta^2}\right)dxds$. This is the first difficulty we need to deal with.

On the other hand, one can deduce from \eqref{Equation-theta} that $\frac 1\theta$ solves
\begin{equation}\label{Identity-temperature-radiative-NS}
c_v\left(\frac 1\theta\right)_t=\left[\frac{\kappa(v,\theta)}{v}\left(\frac 1\theta\right)_x\right]_x
-\frac{2\theta\kappa(v,\theta)}{v}\left[\left(\frac 1\theta\right)_x\right]^2
-\frac{\mu}{v\theta^2}\left[u_x-\frac{R\theta}{2\mu}\right]^2+\frac{R^2}{4\mu v}+\frac{q_x}{\theta^2}.
\end{equation}
Since the appearance of the nonlocal term $\frac{q_x}{\theta^2}$ in the right hand side of \eqref{Identity-temperature-radiative-NS}, how to estimate the lower bound of the absolute temperature in terms of lower bound of the specific volume $v$ and/or the upper bound of the absolute temperature $\theta$ is the second difficulty we need to resolve.

To overcome the first difficulty, if we consider the problem in the one-dimensional periodic box $\mathbb{T}\cong \mathbb{R}/\mathbb{Z}=\left[-\frac 12,\frac 12\right]$, then standard energy estimates can yield the estimates on $\|(v(t),\theta(t))\|_{L^1(\mathbb{T})}$ and $\|u(t)\|_{L^2(\mathbb{T})}$ immediately. Based on these estimates and by employing the identity
\begin{equation}\label{Kanel-identity}
\left(\frac{\mu v_x}{v}\right)_t=u_t+p_x,
\end{equation}
which was first observed by Ja. I. Kanel' in \cite{Kanel-DU-1968} for viscous isentropic flow, we can deduce the desired positive lower bound estimate of the specific volume $v(t,x)$. Even so, we still can not deduce the estimate on $\int^t_0\int_{\mathbb{T}}\left(\frac{\mu u_x^2}{v\theta}+\frac{\kappa(v,\theta)\theta^2_x}{v\theta^2}\right)dxds$ from the basic energy estimates obtained through the normalized entropy $\eta(v,u,\theta)$.
Motivated by the work of B. Kawohl in \cite{Kawohl-JDE-1985} for real gas and Y.-K. Liao and H.-J. Zhao in \cite{Liao-Zhao-JDE-2018} for a viscous radiative and reactive gas, we introduce the auxiliary functions
$\mathfrak{X}(t), \mathfrak{Y}(t),$ and $\mathfrak{Z}(t)$ defined by \eqref{X(t)}, \eqref{Y(t)}, and \eqref{Z(t)}, and by making full use of the intrinsic structure of the system \eqref{Radiation-NS} under our consideration, we can indeed find the relations between $\|v\|_{L^\infty([0,T]\times\mathbb{T})}$ (the upper bound of the specific volume) and $\|\theta\|_{L^\infty([0,T]\times\mathbb{T})}$ (the upper bound on the absolute temperature). As a consequence, we can deduce further the desired upper bound estimates on both the specific volume and the absolute temperature provided that the parameter $\beta$ in \eqref{Kappa} is chosen sufficiently large. It is worth to pointing out that a new iteration argument which is introduced to improve the estimate on $\int^t_0\|\theta(s)\|^m_{L^\infty(\mathbb{T})}ds$ for some positive constant $m>0$ plays an important role in our analysis.

As for the second difficulty, if we go back to the Eulerian coordinates and consider the fourth equation of the system \eqref{1D-Radiative-Hydrodynamics-Euler} in the whole space, one can easily deduce that
\begin{equation}\label{expression-of-q}
q(t,y)=\frac 12\int_{\mathbb{R}}e^{-|y-z|}\sgn(y-z)|\theta(t,z)|^4dz,
\end{equation}
where, without loss of generality, we have set the positive constants $a$, $b$ in \eqref{1D-Radiative-Hydrodynamics-Euler} to be $1$ and $\sgn$ is the sign function:
$$
\sgn(x)=\left\{
\begin{array}{rl}
1,&\quad \textrm{if}\ x>0,\\
0,&\quad \textrm{if}\ x=0,\\
-1,&\quad \textrm{if}\ x<0.
\end{array}
\right.
$$
From \eqref{expression-of-q}, it is easy to see that
$$
\frac{\partial q(t,y)}{\partial y}=|\theta(t,y)|^4-\frac 12\left(\int^y_{-\infty}e^{z-y}|\theta(t,z)|^4dz
+\int^{+\infty}_ye^{y-z}|\theta(t,z)|^4dz\right),
$$
and hence one can deduce the following pointwise estimate between $q_y(t,y)$ and $|\theta(t,y)|^4$:
\begin{equation}\label{pointwise-q-theta-Euler}
\frac{\partial q(t,y)}{\partial y}\leq |\theta(t,y)|^4.
\end{equation}
Based on such an observation, for the periodic initial-boundary value problem \eqref{Radiation-NS}, \eqref{Radiation-NS-IC},
although the last term $\frac{q_x}{\theta^2}$ in the right hand side of \eqref{Identity-temperature-radiative-NS} is a nonlocal term, we can still hope to deduce a pointwise estimate between $q_x(t,x)$ and $|\theta(t,x)|^4$. Such a pointwise estimate together with the upper bounds on both the specific volume and the absolute temperature and the maximum principle for second-order parabolic equations can yield the desired positive lower bound on the absolute temperature.

Before concluding this section, it is worth to pointing out that the existence of global smooth large solutions together with the precise description of their large-time behavior to some initial-boundary value problems in the bounded open interval $(0,1)$ of a compressible, viscous and heat-conducting, one-dimensional monatomic ionized gas are obtained in \cite{Liao-Zhao-Zhou-SIAM-2021} and \cite{Liao-Zhao-2022}, respectively. For such a system, similar to the system \eqref{Radiation-NS} studied in this paper, the dependence of the state of such an ionized gas on the degree of ionization does lead to the loss of concavity of the physical entropy in some small bounded domain and consequently makes it also impossible  to deduce the dissipative estimates on $\int^t_0\int^1_0\left(\frac{\mu(v) u_x^2}{v\theta}+\frac{\kappa(v,\theta)\theta_x^2}{v\theta^2}\right)dxds$ from the basic energy estimates obtained through the normalized entropy $\eta(v, u, \theta)$. The main observation in \cite{Liao-Zhao-2022, Liao-Zhao-Zhou-SIAM-2021} is that, by making full use of the intrinsic structure of the equations under consideration, we can obtain the following estimates
\begin{eqnarray}\label{Ionized-gas-1}
\int^t_0\int^1_0\left(\frac{\mu(v) u_x^2}{v\theta}+\frac{\kappa(v,\theta)\theta_x^2}{v\theta^2}\right)(s,x)dxds
&\leq &C(T)\left(1+\left\|\frac{1}{v\mu(v)}\right\|_{L^\infty([0,T]\times[0,1])}\right),\nonumber\\
 \frac{1}{\theta(t,x)}&\leq& C(T)\left(1+\left\|\frac{1}{v\mu(v)}\right\|_{L^\infty([0,T]\times[0,1])}\right)
\end{eqnarray}
for all $(t,x)\in[0,T]\times[0,1]$ and some positive constant $C(T)$ depending on $T$ and the following an improved estimate on $\int^t_0\int^1_0\left(\frac{\mu(v) u_x^2}{v\theta}+\frac{\kappa(v,\theta)\theta_x^2}{v\theta^2}\right)dxds$
\begin{equation}\label{Ionized-gas-2}
\int^t_0\int^1_0\left(\frac{\mu(v) u_x^2}{v\theta}+\frac{\kappa(v,\theta)\theta_x^2}{v\theta^2}\right)(s,x)dxds
\leq C\left(1+\left\|\frac{1}{v}\right\|^\delta_{L^\infty([0,T]\times[0,1])}\right), \quad 0\leq t\leq T
\end{equation}
for any positive constant $\delta>0$ and some generic positive constant $C>0$ independent of the time variable, from which one can then deduce the global solvability results and then study the large-time behavior of the global smooth large solutions constructed. We note, however, that for the radiation hydrodynamics system \eqref{Radiation-NS} with viscosity and thermal conductivity, even for the initial-boundary value problem \eqref{Radiation-NS}, \eqref{Radiation-NS-IC} in the periodic box $\mathbb{T}$, it seems impossible to deduce the corresponding estimates similar to that of \eqref{Ionized-gas-1} and \eqref{Ionized-gas-2} from the identity \eqref{entropy identity} for the normalized entropy $\eta(v,u,\theta)$.

This paper is organized as follows: After this introduction and the statement of main result, which constitutes Section 1, we prove Theorem \ref{main theorem} in Section 2. To clarify the presentation, the proof of a technical lemma used in Section 2 will be given in the appendix.\\

\noindent \textit{Notations.} For notational simplicity, throughout this paper, $C$ denotes some generic positive constant depending only on the initial data $(v_0(x), u_0(x),\theta_0(x))$ but independent of time $t$, which may take different values in different places. $C(T)$ represents some positive constants depending on the given positive number $T$ and the initial data $(v_0(x), u_0(x),\theta_0(x))$, which may also take different values in different places. $L^q(\mathbb{T})(1\leq q\leq \infty)$ stands for the usual Lebesgue space with norm $\|\cdot\|_{L^q}$,
and $H^k(\mathbb{T})(k\in \mathbb{Z}_+)$ is the usual Sobolev space with norm $\|\cdot\|_{H^k}$. For convenience, we use $\|\cdot\|$ and $\|\cdot\|_k$ to denote $\|\cdot\|_{L^2}$ and $\|\cdot\|_{H^k}$, respectively, in the rest of this paper. It is easy to see that $\|\cdot\|=\|\cdot\|_0$. In addition, we use $\|f\|_{L_{t,x}^\infty}$ to denote the norm of an element $f(t,x)\in L^\infty(0,T;L^\infty(\mathbb{T}))$.

\section{The proof of Theorem \ref{main theorem}}

In this section, we prove Theorem \ref{main theorem}. Our approach is based on the continuation method, which is a combination of local solvability result with certain {\it a priori} estimates on the local solutions constructed.

For the local solvability of the initial-boundary value problem \eqref{Radiation-NS}, \eqref{Radiation-NS-IC}, we have
\begin{lemma}[Local Existence]\label{local existence}
   Assume that the conditions listed in Theorem \ref{main theorem} hold, then there exists a sufficiently small positive constant $t_{0}$ depending on $\left\|\left(v_{0}, u_{0}, \theta_{0}\right)\right\|_2$, $\inf_{x\in\mathbb{T}}\{v_0(x)\}$, and $\inf_{x\in\mathbb{T}}\{\theta_0(x)\}$ such that the initial-boundary value problem \eqref{Radiation-NS}, \eqref{Radiation-NS-IC} in the one-dimensional periodic box $\mathbb{T}$ admits a unique smooth solution $(v(t,x), u(t,x), \theta(t,x), q(t,x))$ defined on the strip $\Pi_{t_{0}}=[0,t_0]\times\mathbb{T}$. Moreover, the solution $(v(t,x), u(t,x), \theta(t,x), q(t,x))$ satisfies
\begin{eqnarray*}
  (v(t,x), u(t,x), \theta(t,x), q(t,x)) &\in& C^{0}\left([0, t_{0}]; H^{1}(\mathbb{T})\right),\\
  \left(u_{xx}(t,x), q_{xx}\right)(t,x) &\in& C^{0}\left([0, t_{0}] ; L^{2}(\mathbb{T})\right),\\
  u_t(t,x)&\in& C^{0}\left([0, t_{0}]; L^{2}(\mathbb{T})\right),\\
  \left(u_{xt}(t,x),\theta_{xx}(t,x)\right)&\in& L^2\left([0, t_{0}]; L^{2}(\mathbb{T})\right),
\end{eqnarray*}
\begin{eqnarray*}
&&\frac{1}{2} \inf _{x \in \mathbb{T}} \{v_{0}(x)\} \leq v(t, x) \leq 2 \sup _{x \in \mathbb{T}} \{v_{0}(x)\},\quad \forall(t, x) \in\left[0, t_{0}\right] \times \mathbb{T},\\
&&\frac{1}{2} \inf _{x \in \mathbb{T}} \{\theta_{0}(x)\} \leq \theta(t, x) \leq 2 \sup _{x \in \mathbb{T}} \{\theta_{0}(x)\} ,\quad \forall(t, x) \in\left[0, t_{0}\right] \times \mathbb{T},
\end{eqnarray*}
and
\begin{equation*}
  \|(v(t), u(t), \theta(t),q(t))\|_2 \leq 2\left\|\left(v_{0}, u_{0}, \theta_{0}\right)\right\|_2, \quad t\in[0,t_0].
\end{equation*}
\end{lemma}
Lemma \ref{local existence} can be easily proved by using a similar method as in \cite{Kazhikhov-Shelukhin-JAMM-1977} for the one-dimensional viscous and heat conducting ideal polytropic gas, \cite{Tani-PRIMS-1977} for the three dimensional compressible viscous fluid motion, or \cite{Kagei-Kawashima-JHDE-2006} for a quasilinear hyperbolic-parabolic system. Here we omit the details for brevity.

Suppose that the unique local solution $(v(t,x), u(t,x), \theta(t,x), q(t,x))$ constructed in Lemma \ref{local existence} has been extended to the strip $\Pi_T=[0,T]\times\mathbb{T}$ for some $T\geq t_0>0$ which satisfies similar estimates listed in Lemma \ref{local existence}, we now turn to deduce certain {\it a priori} estimates in terms of
$\left\|\left(v_{0}, u_{0}, \theta_{0}\right)\right\|_2$, $\inf _{x \in \mathbb{T}}\{v_0(x), \theta_{0}(x)\}$. For result in this direction, we can get the following result:
\begin{lemma}[A priori estimate]\label{a priori estimate}
Assume that the unique local solution $(v(t,x), u(t,x), \theta(t,x), q(t,x))$ constructed in Lemma \ref{local existence} has been extended to the strip $\Pi_T=[0,T]\times\mathbb{T}$ which satisfies similar estimates listed in Lemma \ref{local existence}, then if the assumptions listed in Theorem \ref{main theorem} are satisfied, there exist positive constants $\underline{V}(T)$, $\overline{V}(T)$, $\underline{\Theta}(T)$, $\overline{\Theta}(T)$, and $C(T)$, which depend only on $\left\|\left(v_{0}, u_{0}, \theta_{0}\right)\right\|_2$, $\inf _{x \in \mathbb{T}}\{v_0(x), \theta_{0}(x)\}$, and $T$, such that $(v(t,x), u(t,x), \theta(t,x), q(t,x))$ satisfies the following {\it a priori} estimates:
\begin{eqnarray}\label{A priori esitmiates-on-bounds-of-v-theta}
\underline{V}(T) \leq v(t,x) &\leq& \overline{V}(T),\quad \forall(t, x) \in\left[0, T\right] \times \mathbb{T},\nonumber\\
\underline{\Theta}(T) \leq \theta(t,x) &\leq& \overline{\Theta}(T),\quad \forall(t, x) \in\left[0, T\right] \times \mathbb{T},
\end{eqnarray}
and
\begin{equation}\label{A priori esitmiates-on-H2-norm}
   \|(v(t), u(t), \theta(t), q(t))\|_1^2+\left\|\left(u_{xx}(t),q_{xx}(t),u_t(t)\right)\right\|^2 +\int_{0}^{t}\left\|\left(u_{xt}(s),\theta_{xx}(s)\right)\right\|^2ds\leq C(T),\quad 0\leq t\leq T.
  \end{equation}
\end{lemma}

Having obtained Lemma \ref{a priori estimate}, Theorem \ref{main theorem} can be easily proved by the continuation argument. Thus to prove Theorem \ref{main theorem}, we only need to prove the {\it a priori} estimates \eqref{A priori esitmiates-on-bounds-of-v-theta} and \eqref{A priori esitmiates-on-H2-norm} stated in Lemma \ref{a priori estimate}.

To make the presentation clear, we divide the proof Lemma \ref{a priori estimate} into three parts, the first one is concerned with the positive lower estimate on the specific volume.

\subsection{Estimate on the lower bound of the specific volume}

In this subsection, we try to deduce a positive lower bound of the specific volume $v(t,x)$. For this purpose, we first list the following estimates which is a direct consequence of the conservative structure of the system \eqref{Radiation-NS} and the periodic boundary condition.
\begin{lemma}\label{conservation of mass} Under the assumptions listed in Lemma \ref{a priori estimate}, we have for all $0\leq t\leq T$ that
\begin{equation}\label{conservation of specific volume}
\int_{\mathbb{T}}v(t,x)dx=\int_{\mathbb{T}}v_0(x)dx=1,
\end{equation}
\begin{equation}\label{conservation of velocity}
\int_{\mathbb{T}}u(t,x)dx=\int_{\mathbb{T}}u_0(x)dx,
\end{equation}
\begin{equation}\label{conservation of energy}
  \int_{\mathbb{T}}\left(c_v\theta(t,x)+\frac{1}{2}u(t,x)^2\right)dx=\int_{\mathbb{T}}\left(c_v\theta_0(x)+\frac{1}{2}u_0(x)^2\right)dx,
\end{equation}
and
\begin{equation}\label{consevation of radiation}
  \int_{\mathbb{T}}v(t,x)q(t,x)dx=0.
\end{equation}
\end{lemma}

From \eqref{conservation of specific volume}, we can deduce for each $0\leq t\leq T$ that there exists a $a(t)\in\mathbb{T}=\left[-\frac 12,\frac 12\right]$ such that
\begin{equation*}\label{definition of a(t)}
  v(t,a(t))=\int_{\mathbb{T}}v_0(x)dx=1.
\end{equation*}

Setting
\begin{eqnarray}\label{definition for B(t,x) and Y(t)}
B(t,x)&=&\exp\left(-\frac{1}{\mu}\int_{a(t)}^{x}\left(u(t,y)-u_0(y)\right)dy\right),\\
Y(t)&=&\frac{v(t,a(t))}{v_0(a(t))}\exp\left(\frac{1}{\mu}\int_{0}^{t}\frac{R\theta(s,a(s))}{v(s,a(s))}ds\right),\nonumber
\end{eqnarray}
we can further deduce the following representation formula for the specific volume $v(t,x)$:
\begin{equation}\label{representation for specific volume}
v(t,x)=\frac{v_0(x)}{B(t,x)Y(t)}\left(1+\int_{0}^{t}\frac{R}{\mu}\frac{\theta(s,x)B(s,x)Y(s)}{v_0(x)}ds\right).
\end{equation}

To use the representation formula \eqref{representation for specific volume} for the specific volume to deduce the desired lower bound on $v(t,x)$, we need first to show that both $B(t,x)$ and $Y(t)$ are bounded from below and above. In fact from \eqref{definition for B(t,x) and Y(t)}, \eqref{representation for specific volume}, and the estimates obtained in Lemma \ref{conservation of mass}, we have
\begin{lemma}\label{estimate for B and Y} Under the assumptions listed in Lemma \ref{a priori estimate}, there exists a positive constant $C(T)$ which depends only on initial data $(v_0(x),u_0(x),\theta_0(x))$ and $T$ such that
\begin{equation}\label{Estimate-on-B(t,x)-Y(t)}
    C(T)^{-1}\leq B(t,x)\leq C(T),\quad C(T)^{-1}\leq Y(t)\leq C(T),\quad\forall(t,x)\in\mathbb{T}\times[0,T].
  \end{equation}
\end{lemma}
\begin{proof} The estimate of $B(t,x)$ follows immediately from the estimate \eqref{conservation of energy} obtained in Lemma \ref{conservation of mass} and the H\"{o}lder inequality.

To yield an estimate on $Y(t)$, we get from the representation formula \eqref{representation for specific volume} for the  specific volume that
\begin{equation*}
v(t,x)Y(t)=\frac{v_0(x)}{B(t,x)}\left(1+\int_{0}^{t}\frac{R}{\mu}\frac{\theta(s,x)B(s,x)Y(s)}{v_0(x)}ds\right).
\end{equation*}

Integrating the above identity with respect $x$ over $\mathbb{T}$, we have from the estimate \eqref{conservation of specific volume} obtained in Lemma \ref{conservation of mass} that
\begin{equation*}
Y(t)=\left(\int_{\mathbb{T}}v_0(x)dx\right)^{-1}\left[\int_{\mathbb{T}}\frac{v_0(x)}{B(t,x)}dx +\frac{R}{\mu}\int_{\mathbb{T}}\frac{v_0(x)}{B(t,x)}\left(\int_{0}^{t}\frac{\theta(s,x)B(s,x)Y(s)}{v_0(x)}ds\right)dx\right].
\end{equation*}
Therefore, we can deduce from the estimate of $B(t,x)$, the assumptions imposed on the initial data, and Fubini's theorem that there exists a positive constant $C(T)\geq 1$ depending only on the initial data $(v_0(x), u_0(x), \theta_0(x))$ and $T$ such that
\begin{equation*}
C^{-1}(T)+C^{-1}(T)\int_{0}^{t}\left(\int_{\mathbb{T}}\frac{\theta(s,x)}{v_0(x)}dx\right)Y(s)ds\leq Y(t)\leq C(T)+C(T)\int_{0}^{t}\left(\int_{\mathbb{T}}\frac{\theta(s,x)}{v_0(x)}dx\right)Y(s)ds,
\end{equation*}
from which one can further deduce from the estimate \eqref{conservation of energy} obtained in Lemma \ref{conservation of mass} and the assumptions imposed on the initial specific volume $v_0(x)$ that
\begin{equation*}
C^{-1}(T)+C^{-1}(T)\int_{0}^{t}Y(s)ds\leq Y(t)\leq C(T)+C(T)\int_{0}^{t}Y(s)ds.
\end{equation*}

Having obtained the above inequality, the estimate of $Y(t)$ follows directly from the Gronwall inequality. This completes the proof of Lemma \ref{estimate for B and Y}.
\end{proof}

From the representation formula \eqref{representation for specific volume} and the estimate on the upper bound for $Y(t)$ obtained in Lemma \ref{estimate for B and Y}, we can now derive the lower bound of specific volume.
\begin{lemma}\label{lower bound for specific volume}
Under the assumptions stated in Lemma \ref{a priori estimate}, there exists a positive constant $C(T)$ depending only on initial data $(v_0(x), u_0(x), \theta_0(x))$ and $T$ such that
\begin{equation}\label{Estimate-lower-bound-v}
v(t,x)\geq C(T)
\end{equation}
holds for all $(t,x)\in[0,T]\times\mathbb{T}$.
\end{lemma}
\begin{proof}
Since all the quantities in the right hand side of \eqref{representation for specific volume} are positive, we have
\begin{eqnarray*}
v(t,x)&\geq&\frac{v_0(x)}{B(t,x)Y(t)}\\
 &\geq& C(T)v_0(x) \exp\left(\frac{1}{\mu}\int_{a(t)}^{x}u(t,y)dy\right) \exp\left(-\frac{1}{\mu}\int_{a(t)}^{x}u_0(y)dy\right)\\
 &\geq&C(T)v_0(x)\exp\left(-\frac{2}{\mu}\int_{\mathbb{T}}\left(\frac{R}{\gamma-1}\theta_0(x)+\frac{1}{2}u_0(x)^2\right)dx\right).
\end{eqnarray*}
This completes the proof of Lemma \ref{lower bound for specific volume}.
\end{proof}

\subsection{Estimates on the upper bounds of the specific volume and the absolute temperature}
In this subsection, we deduce the desired upper bounds on both the specific volume $v(t,x)$ and the absolute temperature $\theta(t,x)$. To this end, motivated by the work of B. Kawohl in \cite{Kawohl-JDE-1985} for real gas and Y.-K. Liao and H.-J. Zhao in \cite{Liao-Zhao-JDE-2018} for a viscous radiative and reactive gas, we introduce the following auxiliary functions
\begin{eqnarray}
  \mathfrak{X}(t)&=&\int_{0}^{t}\int_{\mathbb{T}}\frac{\kappa(v(s,x),\theta(s,x))\left|\theta_t(t,x)\right|^2}{v(s,x)}dxds,\label{X(t)}\\
  \mathfrak{Y}(t)&=&\max_{0\leq s\leq t}\left\{\int_{\mathbb{T}}\frac{\left|\kappa(v(s,x),\theta(s,x))\right|^2\left|\theta_x(s,x)\right|^2}{|v(s,x)|^2}dx\right\},\label{Y(t)}\\
  \mathfrak{Z}(t)&=&\max_{0\leq s\leq t}\left\{\int_{\mathbb{T}}\frac{\left|u_{xx}(s,x)\right|^2}{|v(s,x)|^2}dx\right\}.\label{Z(t)}
\end{eqnarray}
It is worth to pointing out that the definitions of the auxiliary functions $\mathfrak{X}(t), \mathfrak{Y}(t)$, and $\mathfrak{Z}(t)$ are a slight modification of those introduced in \cite{Kawohl-JDE-1985, Liao-Zhao-JDE-2018} since in our case we have not obtained the upper bound for the specific volume $v(t,x)$ at this moment.

According to the definitions of $\mathfrak{X}(t)$, $\mathfrak{Y}(t)$, and $\mathfrak{Z}(t)$, we have the following estimate
\begin{lemma}\label{crude estimate for theta}
Under the assumptions listed in Lemma \ref{a priori estimate}, there exists a positive constant $C(T)$ depending only on initial data $(v_0(x), u_0(x), \theta_0(x))$ such that
\begin{equation}\label{Rough-estimate-theta}
\|\theta(t)\|_{L^\infty}\leq C(T)\left(1+\mathfrak{Y}(t)^{\mathfrak{r}}\right),\quad 0\leq t\leq T,
\end{equation}
where $\mathfrak{r}=\frac{1}{2\beta+3}$.
\end{lemma}
\begin{proof}
According to the estimate \eqref{conservation of energy} obtained in Lemma \ref{conservation of mass}, for each $t\in(0,T]$, there exists a $b(t)\in \mathbb{T}$ such that
\begin{equation*}
  c_v\theta(t,b(t))+\frac{1}{2}u^2(t,b(t))=\int_{\mathbb{T}}\left(c_v\theta_0(x)+\frac{1}{2}u_0(x)^2\right)dx.
\end{equation*}
Noticing also
\begin{equation*}
  \left|\theta^m(t,x)-\theta^m(t,b(t))\right|\leq \left|\int_{b(t)}^{x}\left||\theta(t,y)|^{m-1}\theta_x(t,y)\right|dy\right| \leq\int_{\mathbb{T}}\left||\theta(t,x)|^{m-1}\theta_x(t,x)\right|dx,
\end{equation*}
we can get that
\begin{eqnarray*}
\|\theta(t)\|_{L^\infty}^m&\leq&C\left(1+\int_{\mathbb{T}}\left||\theta(t,x)|^{m-1}\theta_x(t,x)\right|dx\right)\\ &\leq&C(T)\left(1+\|\theta(t)\|_{L^\infty}^{\max\left\{0,m-\beta-\frac{3}{2}\right\}} \int_{\mathbb{T}}\left|\frac{\sqrt{\theta(t,x)}\kappa(v(t,x),\theta(t,x))\theta_x(t,x)}{v(t,x)}\right|dx\right)\\
&\leq&C(T)\left(1+\|\theta(t)\|_{L^\infty}^{\max\left\{0,m-\beta-\frac{3}{2}\right\}} \left(\int_{\mathbb{T}}\theta(t,x)dx\right)^\frac{1}{2}\left(\int_{\mathbb{T}}\frac{|\kappa(v(t,x),\theta(t,x))|^2\left|\theta_x(t,x)\right|^2} {|v(t,x)|^2}dx\right)^\frac{1}{2}\right)\\
&\leq&C(T)\left(1+\|\theta(t)\|_{L^\infty}^{\max\left\{0,m-\beta-\frac{3}{2}\right\}}\mathfrak{Y}(t)^\frac{1}{2}\right).
\end{eqnarray*}

If we let $m$ sufficiently large such that $m-\beta-\frac{3}{2}>0$, then for each $\varepsilon>0$, we can get from the above estimate and Young's inequality that
\begin{eqnarray*}
\|\theta(t)\|_{L^\infty}^m&\leq&\varepsilon\|\theta(t)\|_{L^\infty}^{m}+C(T)\left(1+\mathfrak{Y}(t)^\frac{m}{2\beta+3}\right)
\end{eqnarray*}
and the estimate \eqref{Rough-estimate-theta} follows immediately if we choose $\varepsilon>0$ sufficiently small. This completes the proof of Lemma \ref{crude estimate for theta}.
\end{proof}

Having obtained the estimate \eqref{Rough-estimate-theta}, we now turn to deduce some rough estimates on $\|v(t)\|_{L^\infty}$, $\|u_x(t)\|$, and $\|u_x(t)\|_{L^\infty}$. For this purpose, we first estimate $\|v(t)\|_{L^\infty}$ in the following lemma.
\begin{lemma}\label{crude estimate for v}
Under the conditions listed in Lemma \ref{a priori estimate}, there exists a positive constant $C(T)$ depending only on the initial data $(v_0(x), u_0(x), \theta_0(x))$ and $T$ such that
\begin{equation}\label{Rough-estimate-v}
  \|v(t)\|_{L^\infty}\leq C(T)\left(1+\mathfrak{Y}(t)^{\mathfrak{r}}\right),\quad 0\leq t\leq T.
\end{equation}
\end{lemma}
\begin{proof}
By the representation formula \eqref{representation for specific volume} for $v(t,x)$ and the estimate \eqref{estimate for B and Y} obtained in Lemma \ref{estimate for B and Y}, we have for $(t,x)\in[0,T]\times\mathbb{T}$ that
\begin{equation*}
  |v(t,x)|\leq C(T)\left(1+\int_{0}^{t}\|\theta(s)\|_{L^\infty}ds\right)\leq C(T)\left(1+\|\theta\|_{L^\infty_{t,x}}\right)
\end{equation*}
and \eqref{Rough-estimate-v} follows from \eqref{Rough-estimate-theta}.
\end{proof}
\begin{remark} The representation formula \eqref{representation for specific volume} for the specific volume $v(t,x)$ is used in the proofs of Lemmas \ref{estimate for B and Y}, \ref{lower bound for specific volume}, and \ref{crude estimate for v} and to deduce such a representation formula for $v(t,x)$, one needs to assume that the viscosity $\mu$ is a positive constant.

We note, however, that the argument developed in this paper can be adapted to treat the case when the viscosity $\mu>0$ is a smooth function of the specific volume $v$. In fact, from \eqref{Radiation-NS}$_1$ and \eqref{Radiation-NS}$_2$, one can deduce that
\begin{equation}\label{Kanel-identity-2}
\left(\frac{\mu(v)v_x}{v}\right)_t=\left[\int^v_1\frac{\mu(z)}{z}dz\right]_{xt}=u_t+p_x,
\end{equation}
from which and based on some assumptions imposed on the viscosity $\mu(v)$ at $v\to 0_+$ and $v\to+\infty$, one can deduce the following estimates:
\begin{itemize}
\item Firstly, from \eqref{Kanel-identity-2} and the estimates obtained in Lemma \ref{conservation of mass}, one can deduce the desired positive lower estimate on the specific volume $v(t,x)$ as in \cite{Kawohl-JDE-1985};
\item Secondly, from \eqref{Kanel-identity-2}, as in \cite{Tan-Yang-Zhao-Zou-SIMA-2013}, one can estimate the upper bound of the specific volume in terms of $\mathfrak{Y}(t)$.
\end{itemize}
\end{remark}

For the estimates on $\|u_x(t)\|$ and $\|u_x(t)\|_{L^\infty}$, we have
\begin{lemma}\label{crude estimate for ux}
Under the assumptions stated in Lemma \ref{a priori estimate}, there exists a positive constant $C(T)$ depending only on the initial data $(v_0(x), u_0(x), \theta_0(x))$ and $T$ such that
\begin{eqnarray}
\left\|u_x(t)\right\|&\leq& C(T)\left(1+\mathfrak{Y}(t)^{\frac{1}{2}\mathfrak{r}}\right)\mathfrak{Z}(t)^\frac{1}{4},\quad 0\leq t\leq T,\label{Rough-estimate-u-x-L2}\\
\left\|u_x(t)\right\|_{L^\infty}&\leq& C(T)\left(1+\mathfrak{Y}(t)^{\frac{3}{4}\mathfrak{r}}\right)\mathfrak{Z}(t)^\frac{3}{8},\quad 0\leq t\leq T.\label{Rough-estimate-u-x-L-infty}
\end{eqnarray}
\end{lemma}
\begin{proof}
By Sobolev's inequalities, the estimate \eqref{conservation of energy}, the estimate \eqref{Rough-estimate-v} obtained in Lemma \ref{crude estimate for v}, and noticing that
$$
\int_{\mathbb{T}}u_x(t,x)dx=0,
$$
we can get for $0\leq t\leq T$ that
\begin{equation*}
\|u_x(t)\|\leq C\|u(t)\|^\frac{1}{2}\left\|u_{xx}(t)\right\|^\frac{1}{2}\leq C\left\|u_{xx}(t)\right\|^\frac{1}{2}\leq C\|v(t)\|_{L^\infty}^\frac{1}{2}\left\|\frac{u_{xx}(t)}{v(t)}\right\|^\frac{1}{2}
\end{equation*}
and
\begin{equation*}
\|u_x(t)\|_{L^\infty}\leq C\|u_x(t)\|^\frac{1}{2}\left\|u_{xx}(t)\right\|^\frac{1}{2}\leq C\|u_x(t)\|^\frac{1}{2}\|v(t)\|^\frac{1}{2}\left\|\frac{u_{xx}(t)}{v(t)}\right\|^\frac{1}{2}.
\end{equation*}

Having obtained the above estimates, the estimates \eqref{Rough-estimate-u-x-L2} and \eqref{Rough-estimate-u-x-L-infty} follows easily and this completes the proof of Lemma \ref{crude estimate for ux}.
\end{proof}

With the above rough estimates on $\|\theta(t)\|_{L^\infty}$, $\|v(t)\|_{L^\infty}$, $\|u_x(t)\|$, and $\|u_x(t)\|_{L^\infty}$ in hand, we now turn to deduce certain energy type estimates. To this end, we first get from equation \eqref{Radiation-NS}$_2$ for the conservation of momentum that
\begin{lemma}\label{crude unclosed energy estiamte for momentum}
Under the conditions listed in Lemma \ref{a priori estimate}, there exists a positive constant $C(T)$ depending only on the initial data $(v_0(x), u_0(x), \theta_0(x))$ and $T$ such that
\begin{equation}
\|u(t)\|^2+\int_{0}^{t}\left\|\frac{u_x(s)}{\sqrt{v(s)}}\right\|^2ds\leq C(T)\left(1+\mathfrak{Y}(t)^{\mathfrak{r}}\right),\quad 0\leq t\leq T.
\end{equation}
\end{lemma}
\begin{proof}
Multiplying the equation \eqref{Radiation-NS}$_2$ of the conservation of momentum by $u(t,x)$ and integrating it over $[0,t]\times\mathbb{T}$, we have from Cauchy's inequality and the estimates \eqref{Estimate-lower-bound-v} and \eqref{conservation of energy} that
\begin{eqnarray*}
\frac 12\int_{\mathbb{T}}u^2(t,x)dx+\mu\int_{0}^{t}\int_{\mathbb{T}}\frac{u_x^2}{v}dxds
&=&\frac 12\int_{\mathbb{T}}u_0^2(x)dx+R\int_{0}^{t}\int_{\mathbb{T}}\frac{\theta u_x}{v}dxds\\
&\leq&\frac 12 \int_{\mathbb{T}}u_0^2(x)dx +\varepsilon\int_{0}^{t}\int_{\mathbb{T}}\frac{u_x^2}{v}dxds +C_\varepsilon\int_{0}^{t}\int_{\mathbb{T}}\frac{\theta^2}{v}dsds\\
&\leq&\frac 12 \int_{\mathbb{T}}u_0^2(x)dx +\varepsilon\int_{0}^{t}\int_{\mathbb{T}}\frac{u_x^2}{v}dxds+C_\varepsilon(T)\|\theta\|_{L_{t,x}^\infty}.
\end{eqnarray*}

By choosing $\varepsilon$ sufficiently small and from the estimate \eqref{Rough-estimate-theta} obtained in Lemma \ref{crude estimate for theta}, we can prove Lemma \ref{crude unclosed energy estiamte for momentum}.
\end{proof}

From the equation \eqref{Radiation-NS}$_4$ satisfied by the radiation flux $q(t,x)$, we can get the following estimate on $q(t,x)$.
\begin{lemma}\label{crude unclosed energy estimate for radiation flux}
Under the assumptions listed in Lemma \ref{a priori estimate}, there exists a positive constant $C(T)$ depending only on the initial data $(v_0(x), u_0(x), \theta_0(x))$ and $T$ such that
\begin{equation}
\left\|\frac{q_x(t)}{\sqrt{v(t)}}\right\|^2+\left\|\sqrt{v(t)}q(t)\right\|^2\leq C(T)\left(1+\mathfrak{Y}(t)^{8\mathfrak{r}}\right),\quad 0\leq t\leq T.
\end{equation}
\end{lemma}
\begin{proof}
Multiplying the equation \eqref{Radiation-NS}$_4$ satisfied by the radiation flux $q(t,x)$ by $q(t,x)$ and integrating it over $\mathbb{T}$, we can get from Cauchy's inequality that
\begin{eqnarray*}
\int_{\mathbb{T}}\frac{q_x^2(t,x)}{v(t,x)}dx+a\int_{\mathbb{T}}v(t,x)q^2(t,x)dx&=&b\int_{\mathbb{T}}q_x\theta^4dx\\
    &\leq&\varepsilon\int_{\mathbb{T}}\frac{q_x^2(t,x)}{v(t,x)}dx+C_\varepsilon b^2\int_{\mathbb{T}}v(t,x)\theta^8(t,x)dx\\
    &\leq&\varepsilon\int_{\mathbb{T}}\frac{q_x^2(t,x)}{v(t,x)}dx+C_\varepsilon b^2\|\theta(t)\|_{L^\infty}^8\int_{\mathbb{T}}v(t,x)dx.
\end{eqnarray*}

With the above estimate in hand, we can finish the proof of Lemma \ref{crude unclosed energy estimate for radiation flux} by
choosing $\varepsilon$ sufficiently small and by employing the estimate \eqref{conservation of specific volume} obtained in Lemma \ref{conservation of mass} and the estimate \eqref{Rough-estimate-theta} obtained in Lemma \ref{crude estimate for theta}.
\end{proof}

For the $L^2(\mathbb{R})-$estimate on the absolute temperature, we can get the following result.
\begin{lemma}\label{crude unclosed energy estimate for energy}
Under the assumptions stated in Lemma \ref{a priori estimate}, there exists a positive constant $C(T)$ depending only on the initial data $(v_0(x), u_0(x), \theta_0(x))$ and $T$ such that
\begin{equation}\label{Rough-estimate-theta-L2}
\|\theta(t)\|^2+\int_{0}^{t}\left\|\frac{\sqrt{\kappa(v(s),\theta(s))}\theta_x(s)}{\sqrt{v(s)}}\right\|^2ds\leq C(T)\left(1+\mathfrak{Y}(t)^{8\mathfrak{r}}\right),\quad 0\leq t\leq T.
\end{equation}
\end{lemma}
\begin{proof}
From the third equation of \eqref{Radiation-NS} for the conservation of energy, it is easy to see that $\theta(t,x)$ satisfies
  \begin{equation*}
    \frac{R}{\gamma-1}\theta_t-\left(\frac{\kappa(v,\theta)\theta_x}{v}\right)_x=-q_x +\frac{\mu u_x^2}{v}-\frac{R\theta}{v}u_x.
  \end{equation*}
  Multiplying the above equation by $\theta(t,x)$ and integrating it over $[0,t]\times\mathbb{T}$, we have
  \begin{eqnarray}\label{Rough-L2-theta-identity}
    &&\frac{R}{\gamma-1}\int_{\mathbb{T}}\frac{1}{2}\theta^2(t,x)dx+\int_{0}^{t}\int_{\mathbb{T}}\frac{\kappa(v,\theta)\theta_x^2}{v}dxds\nonumber\\
    &=&\frac{R}{\gamma-1}\int_{\mathbb{T}}\frac{1}{2}\theta_0^2(x)dx+R\underbrace{\int_{0}^{t}\int_{\mathbb{T}} \frac{\theta^2}{v}u_xdxds}_{I_1}+\underbrace{\int_{0}^{t}\int_{\mathbb{T}}\theta_xqdxds}_{I_2} +\mu\underbrace{\int_{0}^{t}\int_{\mathbb{T}}\frac{\theta}{v}u_x^2dxds}_{I_3}
  \end{eqnarray}
and the last three terms in the right hand side of the above identity can be estimated as follows by using the H\"{o}lder inequality and the estimates obtained in Lemma \ref{crude estimate for theta} to Lemma \ref{crude unclosed energy estimate for radiation flux}:
  \begin{align*}
    |I_1|\leq &C(T)\|\theta\|_{L^\infty_{t,x}}^\frac{3}{2}\left(\int_{0}^{t}\|\theta(s)\|_{L^1}ds\right)^\frac{1}{2} \left(\int_{0}^{t}\left\|\frac{u_x(s)}{\sqrt{v(s)}}\right\|^2ds\right)^\frac{1}{2}\\
    \leq&C(T)\left(1+\mathfrak{Y}(t)^{\frac{3}{2}\mathfrak{r}}\right)\left(1+\mathfrak{Y}(t)^{\frac{1}{2}\mathfrak{r}}\right)\\
    \leq&C(T)\left(1+\mathfrak{Y}(t)^{2\mathfrak{r}}\right),
  \end{align*}
  \begin{align*}
    |I_2|\leq\varepsilon\int_{0}^{t}\int_{\mathbb{T}}\frac{\theta_x^2}{v}dxds+C_\varepsilon\int_{0}^{t}\int_{\mathbb{T}}vq^2dxds,
  \end{align*}
and
  \begin{align*}
    |I_3|\leq&C(T)\|\theta\|_{L^\infty_{t,x}}\int_{0}^{t}\left\|\frac{u_x(s)}{\sqrt{v(s)}}\right\|^2ds\\
    \leq&C(T)\left(1+\mathfrak{Y}(t)^{\mathfrak{r}}\right)\left(1+\mathfrak{Y}(t)^{\mathfrak{r}}\right)\\
    \leq&C(T)\left(1+\mathfrak{Y}(t)^{2\mathfrak{r}}\right).
  \end{align*}

Inserting the above estimates into \eqref{Rough-L2-theta-identity}, we can finish the proof of Lemma \ref{crude unclosed energy estimate for energy} by choosing $\varepsilon>0$ sufficiently small.
\end{proof}

Although Lemmas \ref{crude estimate for theta}-\ref{crude unclosed energy estimate for energy} give some rough estimates on
$\|(v(t),\theta(t))\|_{L^\infty}$, $\|u_x(t)\|_{L^\infty}$, $\|(v(t), u(t), \theta(t))\|$, and $\|(u_x(t),q_x(t))\|$ in terms of $\mathfrak{Y}(t)$ and $\mathfrak{Z}(t)$, these estimates are not sufficient to deduce the desired upper bounds on both the specific volume and the absolute temperature. To yield some refined estimates on these terms, we first deduce the following $L^p([0,T]; L^\infty(\mathbb{T}))-$estimate on $\theta(t,x)$ for $p\geq 1$.
\begin{lemma}\label{Lp-L-infty-theta}
Under the assumptions stated in Lemma \ref{a priori estimate}, for arbitrary $1\leq p\leq\beta+3$, there exists a positive constant $C(T)$ depending only on the initial data $(v_0(x), u_0(x), \theta_0(x))$ and $T$ such that
\begin{equation}\label{Estimate-Lp-L-infty-theta}
\int_{0}^{t}\|\theta(s)\|_{L^\infty}^pds\leq C(T)\left(1+\mathfrak{Y}(t)^{\mathfrak{R}(p)}\right),
\end{equation}
where $\mathfrak{R}(p)=\frac{8p}{(2\beta+3)(\beta+3)}$. It is worth to pointing out that if we choose $\beta$ sufficiently large, $\frac{\mathfrak{R}(p)}{p}=\frac{8}{(2\beta+3)(\beta+3)}$ is much smaller than $\mathfrak{r}=\frac{1}{2\beta+3}$. Such a fact plays an important role in our analysis.
\end{lemma}
\begin{proof}
Noticing that
\begin{equation*}
\|\theta(t)\|_{L^\infty}^{\beta+3}\leq C\left(1+\int_{\mathbb{T}}\left|\theta^{\beta+2}(t,x)\theta_x(t,x)\right|dx\right)
\end{equation*}
and by using H\"{o}lder's inequality, we have
  \begin{align*}
    \int_{0}^{t}\|\theta(s)\|_{L^\infty}^{\beta+3}ds\leq&C(T)\left(1+\int_{0}^{t}\int_{\mathbb{T}}\left|\theta^{\beta+2}(s,x)\theta_x(s,x)\right|dxds\right)\\
    \leq&C(T)\left(1+\int_{0}^{t}\|\theta(s)\|_{L^\infty}^{\frac{\beta+3}{2}} \left\|\sqrt{\frac{\theta(s)\kappa(v(s),\theta(s))}{v(s)}}\theta_x(s)\right\|_{L^1}ds\right)\\
    \leq&C(T)\left(1+\varepsilon\int_{0}^{t}\|\theta(s)\|_{L^\infty}^{\beta+3}ds +\int_{0}^{t} \left\|\sqrt{\frac{\kappa(v(s),\theta(s))}{v(s)}}\theta_x(s)\right\|^2ds\right).
  \end{align*}

We have from the estimate \eqref{Rough-estimate-theta-L2} obtained in Lemma \ref{crude unclosed energy estimate for energy} that
    \begin{align*}
    \int_{0}^{t}\|\theta(s)\|_{L^\infty}^{\beta+3}ds
    \leq&C(T)\left(1+\varepsilon\int_{0}^{t}\|\theta(s)\|_{L^\infty}^{\beta+3}ds+\mathfrak{Y}(t)^{8\mathfrak{r}}\right).
  \end{align*}

With the above estimates in hand, we can deduce \eqref{Estimate-Lp-L-infty-theta} easily by choosing $\varepsilon>0$ sufficiently small and from the H\"{o}lder inequality. This completes the proof of Lemma \ref{Lp-L-infty-theta}.
\end{proof}

Having obtained the estimate \eqref{Estimate-Lp-L-infty-theta}, we can deduce the following refined estimates on
$\|(u(t), \theta(t), u_x(t))\|$, $\|(v(t), u_x(t))\|_{L^\infty}$, and $\int^t_0\left\|\left(q_x(s)/\sqrt{v(s)},\sqrt{v(s)}q(s)
\right)\right\|^2ds$ in terms of $\mathfrak{Y}(t)$ and $\mathfrak{Z}(t)$.
\begin{lemma}\label{Refined estimates}
Under the assumptions listed in Lemma \ref{a priori estimate}, then for each $0\leq t\leq T$, there exists a positive constant $C(T)$ depending only the initial data $(v_0(x), u_0(x),\theta_0(x))$ and $T$ such that
\begin{eqnarray}
\|v(t)\|_{L^\infty}&\leq& C(T)\left(1+\mathfrak{Y}(t)^{\mathfrak{R}(1)}\right),\label{improved-esitmate-v-L-infty}\\
\|u_x(t)\|&\leq& C(T)\left(1+\mathfrak{Y}(t)^{\mathfrak{R}(\frac{1}{2})}\right)\mathfrak{Z}(t)^\frac{1}{4},\label{improved-estimate-ux-L2}\\
\|u_x(t)\|_{L^\infty}&\leq& C(T)\left(1+\mathfrak{Y}(t)^{\mathfrak{R}(\frac{3}{4})}\right)\mathfrak{Z}(t)^\frac{3}{8}, \label{improved-estimate-ux-L-infty}\\
\|u(t)\|^2+\int_{0}^{t}\left\|\frac{u_x(s)}{\sqrt{v(s)}}\right\|^2ds&\leq& C(T)\left(1+\mathfrak{Y}(t)^{\mathfrak{R}(1)}\right),\label{improved-estimate-u-L2}\\
\int_{0}^{t}\left\|\left(\frac{q_x(s)}{\sqrt{v(s)}},\sqrt{v(s)}q(s)\right)\right\|^2ds&\leq& C(T)\left(1+\mathfrak{Y}(t)^{\mathfrak{R}(8)}\right), \label{improbed-estimate-q}\\
 \|\theta(t)\|^2+\int_{0}^{t}\left\|\sqrt{\frac{\kappa(v(s),\theta(s))}{v(s)}}\theta_x(s)\right\|^2ds&\leq& C(T)\left(1+\mathfrak{Y}(t)^{\mathfrak{R}_1}\right),\label{improved-estimate-theta-L2}
\end{eqnarray}
where $\mathfrak{R}_1=\max\{\mathfrak{r}+\mathfrak{R}(1),\mathfrak{R}(8)\}$.
\end{lemma}
\begin{proof} For the case $\beta\geq 5$, noticing that
\begin{eqnarray*}
\|v(t)\|_{L^\infty}&\leq& C(T)\left(1+\int^t_0\|\theta(s)\|_{L^\infty}ds\right),\\
\|u_x(t)\|&\leq&C\|v(t)\|_{L^\infty}^\frac{1}{2}\left\|\frac{u_{xx}(t)}{v(t)}\right\|^\frac{1}{2},\\
\|u_x(t)\|_{L^\infty}&\leq&C\|u_x(t)\|^\frac{1}{2}\|v(t)\|^\frac{1}{2}\left\|\frac{u_{xx}(t)}{v(t)}\right\|^\frac{1}{2},\\
\int_0^t\int_{\mathbb{T}}\frac{\theta^2}{v}dxds &\leq&C(T)\int_{0}^{t}\|\theta(s)\|_{L^\infty}\left(\int_{\mathbb{T}}\theta(s,x)dx\right)ds\\
&\leq&C(T)\int_{0}^{t}\|\theta(s)\|_{L^\infty}ds,\\
\int_{0}^{t}\int_{\mathbb{T}}v\theta^8dxds&\leq&\int_{0}^{t}\|\theta(s)\|_{L^\infty}^8\left(\int_{\mathbb{T}}v(s,x)dx\right)ds\\
&\leq&C(T)\int_{0}^{t}\|\theta(s)\|_{L^\infty}^8ds,
\end{eqnarray*}
we can obtain the estimates \eqref{improved-esitmate-v-L-infty}, \eqref{improved-estimate-ux-L2}, \eqref{improved-estimate-ux-L-infty}, \eqref{improved-estimate-u-L2}, and \eqref{improbed-estimate-q} by repeating the arguments used in the proofs of Lemmas \ref{crude estimate for v}-\ref{crude unclosed energy estimate for radiation flux}
 and the estimate \eqref{Estimate-Lp-L-infty-theta} obtained in Lemma \ref{Lp-L-infty-theta}.

To complete the proof of Lemma \ref{Refined estimates}, we only need to prove \eqref{improved-estimate-theta-L2}. To do so, we get from the H\"{o}lder's inequality and the estimate \eqref{Estimate-Lp-L-infty-theta} obtained in Lemma \ref{Lp-L-infty-theta} that terms $I_i (i=1,2,3)$ in Lemma \ref{crude unclosed energy estimate for energy} can be estimates now as follows:
  \begin{align*}
    |I_1|\leq &C(T)\left(\int_{0}^{t}\|\theta(s)\|_{L^\infty}^3\left(\int_{\mathbb{T}}\theta(s,x)dx\right)ds\right)^\frac{1}{2} \left(\int_{0}^{t}\left\|\frac{u_x(s)}{\sqrt{v(s)}}\right\|^2ds\right)^\frac{1}{2}\\
    \leq&C(T)\left(1+\mathfrak{Y}(t)^{\mathfrak{R}(\frac{3}{2})}\right)\left(1+\mathfrak{Y}(t)^{\mathfrak{R}(\frac{1}{2})}\right)\\
    \leq&C(T)\left(1+\mathfrak{Y}(t)^{\mathfrak{R}(2)}\right),
  \end{align*}
  \begin{align*}
    |I_2|\leq\varepsilon\int_{0}^{t}\int_{\mathbb{T}}\frac{\theta_x^2}{v}dxds+C_\varepsilon\int_{0}^{t}\int_{\mathbb{T}}vq^2dxds,
  \end{align*}
and
  \begin{align*}
    |I_3|\leq&C(T)\|\theta\|_{L^\infty_{t,x}}\int_{0}^{t}\left\|\frac{u_x(s)}{\sqrt{v(s)}}\right\|^2ds\\
    \leq&C(T)\left(1+\mathfrak{Y}(t)^{\mathfrak{r}}\right)\left(1+\mathfrak{Y}(t)^{\mathfrak{R}(1)}\right)\\
    \leq&C(T)\left(1+\mathfrak{Y}(t)^{\mathfrak{r}+\mathfrak{R}(1)}\right).
  \end{align*}
With these estimates in hand, \eqref{improved-estimate-theta-L2} can be obtained easily. This completes the proof of Lemma \ref{Refined estimates}.
\end{proof}

Based on the above estimates, w now turn to deduce certain estimates on $\|v_x(t)\|$ and $\|u_t(t)\|$ in terms of
$\mathfrak{X}(t)$ and $\mathfrak{Y}(t)$. For the estimate of $\|v_x(t)\|$, we have
\begin{lemma}\label{estimate for vx}
Under the assumptions listed in Lemma \ref{a priori estimate}, we have for arbitrary $n\geq 8$ and $0\leq t\leq T$ that
  \begin{equation}\label{estimate-vx-L2}
    \|v_x(t)\|\leq C(T)\left(1+\mathfrak{Y}(t)^{\frac{\mathfrak{R}_1}{2}+\mathfrak{R}(\frac{1}{2})}\right).
  \end{equation}
Here $C(T)$ is some positive constant which depends only on the initial data $(v_0(x), u_0(x), \theta_0(x))$ and $T$.
\end{lemma}
\begin{proof} From the representation formula \eqref{representation for specific volume} for the specific volume $v(t,x)$, we can get that
  \begin{eqnarray*}
    v_x(t,x)&=&-\frac{1}{v^2(t,x)}\bigg[\frac{u_0(x)-u(t,x)}{v(t,x)}\\
    &&+\frac{1}{v^2(t,x)B(t,x)Y(t)}\left(v_{0x}(x)-\frac{R}{\mu}\int_{0}^{t}B(s,x)Y(s)(\theta_x(s,x)+\theta(t,x)(u_0(x)-u(t,x)))ds\right)\bigg],
  \end{eqnarray*}
 from which we can deduce that
  \begin{equation*}
  \begin{aligned}
    \left\|v_x(t)\right\|^2\leq& C(T)\int_{\mathbb{T}}|u_0(x)-u(t,x)|^2dx+C(T)\|v_{0x}\|^2+C(T)\int_{0}^{t}\|\theta_x(s)\|^2ds\\
    &+C(T)\int_{0}^{t}\int_{\mathbb{T}}|\theta(t,x)(u_0(x)-u(s,x))|^2dxds\\
    \leq&C(T)\Bigg(\|u_0\|^2+\|u(t)\|^2+\|v_{0x}\|^2+\|v\|_{L_{t,x}^\infty}\int_{0}^{t}\left\|\sqrt{\frac{\kappa(v(s),\theta(s))}{v(s)}}\theta_x(s)\right\|^2ds\\
    &+\|v\|_{L_{t,x}^\infty}\int_{0}^{t}\|\theta(s)\|\left\|\sqrt{\frac{\kappa(v(s),\theta(s))}{v(s)}}\theta_x(s)\right\|(\|u_0\|^2+\|u(s)\|^2)ds\Bigg).
  \end{aligned}
  \end{equation*}
Thus \eqref{estimate-vx-L2} follows directly from Lemma \ref{Refined estimates}.
\end{proof}

As to the estimate of $\|u_t(t)\|$, we have
\begin{lemma}\label{esitmate for ut uxt}
Under the conditions listed in Lemma \ref{a priori estimate}, we have for arbitrary $n\geq 1$ and $0\leq t\leq T$ that
  \begin{equation}\label{estimate-ut-L2}
    \|u_t(t)\|^2+\int_{0}^{t}\left\|\frac{u_{xt}(s)}{\sqrt{v(s)}}\right\|^2ds\leq \mathfrak{Z}(t)^{\mathfrak{R}_2}+C(T)\left(1+\mathfrak{X}(t)+\mathfrak{Y}(t)+\mathfrak{Y}(t)^{2\mathfrak{r}+\mathfrak{R}(1)}\right),
  \end{equation}
where $\mathfrak{R}_2=\frac{3}{4}\frac{1}{1-\mathfrak{R}(\frac{5}{2})}$, and $C(T)$ is a positive constant depending only on the initial data  $(v_0(x), u_0(x), \theta_0(x))$ and $T$.
\end{lemma}
\begin{proof}
  Differentiating both sides of the equation \eqref{Radiation-NS}$_2$ for the conservation of momentum with respect of $t$, multiplying the resultant equation by $u_t(t,x)$, and then integrating it over $[0,t]\times\mathbb{T}$, we have
  \begin{eqnarray}\label{ux-L2}
    &&\int_{\mathbb{T}}\frac{1}{2}u_t^2(t,x)dx+\mu\int_{0}^{t}\int_{\mathbb{T}}\frac{u_{xt}^2}{v}dxds\nonumber\\
    &=&\int_{\mathbb{T}}\frac{1}{2}u_{0t}^2(x)dx+\mu\underbrace{\int_{0}^{t}\int_{\mathbb{T}}\frac{u_x^2u_{xt}}{v}dxds}_{J_1}\\
    &&+R\underbrace{\int_{0}^{t}\int_{\mathbb{T}}\frac{\theta_tu_{xt}}{v}dxds}_{J_2} -R\underbrace{\int_{0}^{t}\int_{\mathbb{T}}\frac{\theta u_xu_{xt}}{v^2}dxds}_{J_3}.\nonumber
  \end{eqnarray}
The last three terms in the right hand side of the above identity can be controlled as follows by using H\"{o}lder's inequality and the estimates obtained in Lemma \ref{Refined estimates}:
  \begin{align*}
    |J_1|\leq&\varepsilon\int_{0}^{t}\int_{\mathbb{T}}\frac{u_{xt}^2}{v}dxds +C_\varepsilon\int_{0}^{t}\int_{\mathbb{T}}\frac{u_x^4}{v}dxds\\
    \leq&\varepsilon\int_{0}^{t}\int_{\mathbb{T}}\frac{u_{xt}^2}{v}dxds +C_\varepsilon\|u_x\|_{L_{t,x}^\infty}^2\int_{0}^{t}\left\|\frac{u_x(s)}{\sqrt{v(s)}}\right\|^2ds\\
    \leq&\varepsilon\int_{0}^{t}\int_{\mathbb{T}}\frac{u_{xt}^2}{v}dxds +C_\varepsilon(T)\left(1+\mathfrak{Y}(t)^{\mathfrak{R}(\frac{3}{2})}\right) \mathfrak{Z}(t)^\frac{3}{4}\left(1+\mathfrak{Y}(t)^{\mathfrak{R}(1)}\right)\\
    \leq&\varepsilon\int_{0}^{t}\int_{\mathbb{T}}\frac{u_{xt}^2}{v}dxds +C_\varepsilon(T)\left(1+\mathfrak{Y}(t)^{\mathfrak{R}(\frac{5}{2})}\right)\mathfrak{Z}(t)^\frac{3}{4},
  \end{align*}
  \begin{align*}
    |J_2|\leq&\varepsilon\int_{0}^{t}\int_{\mathbb{T}}\frac{u_{xt}^2}{v}dxds +C_\varepsilon\int_{0}^{t}\int_{\mathbb{T}}\frac{\theta_t^2}{v}dxds\\
    \leq&\varepsilon\int_{0}^{t}\int_{\mathbb{T}}\frac{u_{xt}^2}{v}dxds+C_\varepsilon \mathfrak{X}(t),
  \end{align*}
and
  \begin{align*}
    |J_3|\leq&\varepsilon\int_{0}^{t}\int_{\mathbb{T}}\frac{u_{xt}^2}{v}dxds +C_\varepsilon\int_{0}^{t}\int_{\mathbb{T}}\frac{\theta^2u_x^2}{v}dxds\\
    \leq&\varepsilon\int_{0}^{t}\int_{\mathbb{T}}\frac{u_{xt}^2}{v}dxds +C_\varepsilon\|\theta\|_{L_{t,x}^\infty}^2\int_{0}^{t}\left\|\frac{u_x(s)}{\sqrt{v(s)}}\right\|^2ds\\
    \leq&\varepsilon\int_{0}^{t}\int_{\mathbb{T}}\frac{u_{xt}^2}{v}dxds +C_\varepsilon(T)\left(1+\mathfrak{Y}(t)^{2\mathfrak{r}}\right)\left(1+\mathfrak{Y}(t)^{\mathfrak{R}(1)}\right)\\
    \leq&\varepsilon\int_{0}^{t}\int_{\mathbb{T}}\frac{u_{xt}^2}{v}dxds +C_\varepsilon(T)\left(1+\mathfrak{Y}(t)^{2\mathfrak{r}+\mathfrak{R}(1)}\right).
  \end{align*}

Inserting the above estimates into \eqref{ux-L2} yields \eqref{estimate-ut-L2}. This completes the proof of Lemma \ref{esitmate for ut uxt}.
\end{proof}

Now we turn to derive the upper bounds on both the specific volume and the absolute temperature. For this purpose, we need to find the relations between the auxiliary functions $\mathfrak{X}(t)$, $\mathfrak{Y}(t)$, and $\mathfrak{Z}(t)$. Firstly, we control $\mathfrak{Z}(t)$ in terms of $\mathfrak{X}(t)$ and $\mathfrak{Y}(t)$ as in the following lemma.
\begin{lemma}\label{estimate for Z}
Under the conditions listed in Lemma \ref{a priori estimate}, we can deduce for arbitrary $n\geq 8$ and $0\leq t\leq T$ that there exists a positive constant $C(T)$ depending only on the initial data $(v_0(x), u_0(x), \theta_0(x))$ and $T$ such that
  \begin{equation}\label{estimate-Z(t)}
    \mathfrak{Z}(t)\leq C(T)\left(\mathfrak{X}(t)+\mathfrak{Y}(t)+\mathfrak{Y}(t)^{2\mathfrak{r}+\mathfrak{R}_1+\mathfrak{R}(1)}+\mathfrak{Y}(t)^{2\mathfrak{r}+\mathfrak{R}(1)}+\mathfrak{Z}(t)^{\mathfrak{R}_3}+\mathfrak{Z}(t)^{\mathfrak{R}_2}\right),
  \end{equation}
  where $\mathfrak{R}_3=\frac{3}{4}\frac{1}{1-\mathfrak{R}(\frac{5}{2})-\mathfrak{R}_1}$.
\end{lemma}
\begin{proof}
By the equation \eqref{Radiation-NS}$_2$ for the conservation of momentum, we have
  \begin{equation*}
    \frac{u_{xx}}{v}=\frac{1}{\mu}\left(u_t+R\frac{\theta_x}{v}-R\frac{\theta v_x}{v^2}+\mu\frac{u_xv_x}{v^2}\right),
  \end{equation*}
from which one can deduce that
  \begin{equation*}
    \left\|\frac{u_{xx}(t)}{v(t)}\right\|^2\leq C\left(\|u_t(t)\|^2 +\left\|\frac{\theta_x(t)}{v(t)}\right\|^2+\left\|\frac{\theta(t) v_x(t)}{v^2(t)}\right\|^2+\left\|\frac{u_x(t)v_x(t)}{v^2(t)}\right\|^2\right).
  \end{equation*}

By using Holder inequality, the estimates obtained in Lemma \ref{crude estimate for theta}, Lemma \ref{Refined estimates}, and Lemma \ref{estimate for vx}, we find that
  \begin{align*}
    \left\|\frac{\theta(t) v_x(t)}{v^2(t)}\right\|^2\leq& C(T)\|\theta(t)\|_{L^\infty}^2\|v_x(t)\|^2\\
    \leq&C(T)\left(1+\mathfrak{Y}(t)^{2\mathfrak{r}}\right)\left(1+\mathfrak{Y}(t)^{\mathfrak{R}_1+\mathfrak{R}(1)}\right)\\
    \leq&C(T)\left(1+\mathfrak{Y}(t)^{2\mathfrak{r}+\mathfrak{R}_1+\mathfrak{R}(1)}\right)
  \end{align*}
and
  \begin{align*}
    \left\|\frac{u_x(t)v_x(t)}{v^2(t)}\right\|^2\leq& C(T)\|u_x(t)\|_{L^\infty}^2\|v_x(t)\|^2\\
    \leq&C(T)\left(1+\mathfrak{Y}(t)^{\mathfrak{R}(\frac{3}{2})}\right)\mathfrak{Z}(t)^\frac{3}{4}\left(1+\mathfrak{Y}(t)^{\mathfrak{R}_1+\mathfrak{R}(1)}\right)\\
    \leq&C(T)\left(1+\mathfrak{Y}(t)^{\mathfrak{R}(\frac{5}{2}) +\mathfrak{R}_1}\right)\mathfrak{Z}(t)^\frac{3}{4},
  \end{align*}
from which we can then deduce \eqref{estimate-Z(t)} and the proof of Lemma \ref{estimate for Z} is complete.
\end{proof}

The next lemma shows that $\mathfrak{X}(t)$ and $\mathfrak{Y}(t)$ can be estimated by $\mathfrak{Z}(t)$.
\begin{lemma}\label{estimate for X Y}
Under the conditions listed in Lemma \ref{a priori estimate}, we can deduce for arbitrary $n\geq 8$ and $0\leq t\leq T$ that there exists a positive constant $C(T)$ depending only on the initial data $(v_0(x), u_0(x), \theta_0(x))$ and $T$ such that
  \begin{equation}\label{estimate-X(t)+Y(t)}
    \mathfrak{X}(t)+\mathfrak{Y}(t)\leq C(T)\left(1+\mathfrak{Y}(t)^{\mathfrak{R}_4}+\mathfrak{Z}(t)^{\mathfrak{R}_5}\right),
  \end{equation}
where
   \begin{equation}
  \begin{split}
    \mathfrak{R}_4=\max\Bigg\{&{(2\beta+2)\mathfrak{r}+\mathfrak{R}(2)}, {(\beta+1)\mathfrak{r}+\mathfrak{R}(1)+\frac{1}{2}}, {(\beta+3)\mathfrak{r}+\mathfrak{R}(2)},  {\mathfrak{r}+\mathfrak{R}\left(\frac{9}{2}\right)},\\
    &{\beta\mathfrak{r}+\mathfrak{R}(2)+3\mathfrak{R}_1}, {\frac{\beta+2}{2}\mathfrak{r}+\mathfrak{R}(1)+\frac{3\mathfrak{R}_1}{2}}, {\frac{\beta}{2}\mathfrak{r}+\mathfrak{R}\left(\frac{11}{2}\right)+\frac{3\mathfrak{R}_1}{2}}\Bigg\},
  \end{split}
  \end{equation}
and
  \begin{equation}
  \begin{split}
    \mathfrak{R}_5=\max\Bigg\{&\frac{1}{2}\frac{1}{1-\mathfrak{r}-\mathfrak{R}(\frac{1}{2})}, \frac{1}{2}\frac{1}{1-\frac{\beta+2}{2}\mathfrak{r}-\mathfrak{R}(\frac{1}{2})}, {\frac{\mathfrak{R}_2}{2}}\frac{1}{1-{(\beta+1)\mathfrak{r}-\mathfrak{R}(1)}}, \frac{3}{8}\frac{1}{1-{(\beta+1)\mathfrak{r}-\mathfrak{R}(\frac{7}{4})}},\\
    &\frac{3}{8}\frac{1}{1-\mathfrak{r}-\mathfrak{R}(\frac{7}{4})}, \frac{3}{8}\frac{1}{1-{\mathfrak{R}(\frac{3}{4})-\mathfrak{R}_1}}, \frac{3}{8}\frac{1}{1-\frac{\beta}{2}\mathfrak{r}-{\mathfrak{R}(\frac{7}{4})-\frac{3\mathfrak{R}_1}{2}}}, \frac{1}{2}\frac{1}{1-{\frac{\beta}{2}\mathfrak{r}-\mathfrak{R}(\frac{9}{2})}}\Bigg\}.
  \end{split}
  \end{equation}
\end{lemma}
\begin{proof}
As in \cite{Kawohl-JDE-1985}, we set
  \begin{equation*}
    K(v,\theta)=\int_{0}^{\theta}\frac{\kappa(v,\xi)}{v}d\xi=\frac{\kappa_1\theta}{v}+\frac{\kappa_2\theta^{\beta+1}}{\beta+1},
  \end{equation*}
  then
  \begin{align*}
    K_t(v,\theta)=&K_v(v,\theta)u_x+\frac{\kappa(v,\theta)\theta_t}{v},\\
    K_{xt}(v,\theta)=&\left(\frac{\kappa(v,\theta)\theta_x}{v}\right)_t+K_v(v,\theta)u_{xx}+K_{vv}(v,\theta)v_xu_x+\left(\frac{\kappa(v,\theta)}{v}\right)_vv_x\theta_t,\\
    K_v(v,\theta)=&-\frac{\kappa_1\theta}{v^2},\\
    K_{vv}(v,\theta)=&\frac{2\kappa_1\theta}{v^3}.
  \end{align*}

Multiplying the equation \eqref{Radiation-NS}$_3$ for the conservation of energy by $K_t$ and integrating the resulting identity over $(0,t)\times\mathbb{T}$, we have
  \begin{align*}
    &\int_{0}^{t}\int_{\mathbb{T}}\left(\frac{R}{\gamma-1}\theta_t+R\frac{\theta}{v}u_x-\mu\frac{u_x^2}{v}\right)K_t(v,\theta)dxds +\int_{0}^{t}\int_{\mathbb{T}}\frac{\kappa(v,\theta)}{v}\theta_xK_{xt}(v,\theta)dxds \\
    =&-\int_{0}^{t}\int_{\mathbb{T}}q_xK_t(v,\theta)dxds.
  \end{align*}

The above identity together with the definition of $K(v,\theta)$ yield
\begin{align}\label{estimate-X+Y}
    &\int_{\mathbb{T}}\frac{1}{2}\frac{\kappa^2(v,\theta)\theta_x^2}{v^2}dx +\int_{0}^{t}\int_{\mathbb{T}}\frac{R}{\gamma-1}\frac{\kappa(v,\theta)\theta_t^2}{v}dxds\nonumber\\
    =&\int_{\mathbb{T}}\frac{1}{2}\frac{\kappa(v_0,\theta_0)\theta_{0x}^2}{v_0^2}dx -\int_{\mathbb{T}}\mu\frac{K(v_0,\theta_0)u_{0x}^2}{v_0}dx +\frac{R}{\gamma-1}\underbrace{\int_{0}^{t}\int_{\mathbb{T}}K_v(v,\theta)u_x\theta_tdxds}_{K_1}\nonumber\\
    &+R\underbrace{\int_{0}^{t}\int_{\mathbb{T}}\frac{K_v(v,\theta)\theta u_x^2}{v}dxds}_{K_2}+R\underbrace{\int_{0}^{t}\int_{\mathbb{T}}\frac{\kappa(v,\theta)\theta\theta_tu_x}{v^2}dxds}_{K_3} +\mu\underbrace{\int_{\mathbb{T}}\frac{K(v,\theta)u_x^2}{v}dx}_{K_4}\nonumber\\
    &-\mu\underbrace{\int_{0}^{t}\int_{\mathbb{T}}\frac{2K(v,\theta)u_xu_{xt}}{v}dxds}_{K_5} +\mu\underbrace{\int_{0}^{t}\int_{\mathbb{T}}\frac{K(v,\theta)u_x^3}{v^2}dxds}_{K_6}\\
    &+\frac{R}{\gamma-1}\underbrace{\int_{0}^{t}\int_{\mathbb{T}}K_v(v,\theta)u_{x}\theta_tdxds}_{K_7} +R\underbrace{\int_{0}^{t}\int_{\mathbb{T}}K_v(v,\theta)\frac{\theta u_x^2}{v}dxds}_{K_8} +\underbrace{\int_{0}^{t}\int_{\mathbb{T}}K_v(v,\theta)u_xq_xdxds}_{K_9}\nonumber\\
    &-\mu\underbrace{\int_{0}^{t}\int_{\mathbb{T}}K_v(v,\theta)\frac{u_x^3}{v}dxds}_{K_{10}} +\underbrace{\int_{0}^{t}\int_{\mathbb{T}}\left(\frac{\kappa(v,\theta)}{v}\right)_v \frac{\kappa(v,\theta)u_x\theta_x^2}{v}dxds}_{K_{11}}\nonumber\\
    &-\underbrace{\int_{0}^{t}\int_{\mathbb{T}}\left(\frac{\kappa(v,\theta)}{v}\right)_v\frac{\kappa(v,\theta)v_x\theta_x \theta_t}{v}dxds}_{K_{12}} -\underbrace{\int_{0}^{t}\int_{\mathbb{T}}K_v(v,\theta)u_xq_xdxds}_{K_{13}} -\underbrace{\int_{0}^{t}\int_{\mathbb{T}}\frac{\kappa(v,\theta)\theta_tq_x}{v}dxds}_{K_{14}}.\nonumber
  \end{align}
The terms $K_i(i=1,\cdots,14)$ can be controlled as follows by using H\"{o}lder's inequality, Sobolev's inequalities, and the estimates obtained in Lemma \ref{Refined estimates} to Lemma \ref{esitmate for ut uxt}:
  \begin{align*}
    |K_1|\leq&C(T)\|\theta\|_{L_{t,x}^\infty}\left(\int_{0}^{t}\left\|\frac{u_x(s)}{\sqrt{v(s)}}\right\|^2ds \right)^\frac{1}{2}\left(\int_{0}^{t}\left\|\frac{\theta_t(s)}{\sqrt{v(s)}}\right\|ds\right)^\frac{1}{2}\\
    \leq&C(T)\left(1+\mathfrak{Y}(t)^{\mathfrak{r}}\right)\left(1+\mathfrak{Y}(t)^{\mathfrak{R}(\frac{1}{2})}\right)\mathfrak{Z}(t)^\frac{1}{2}\\
    \leq&C(T)\left(1+\mathfrak{Y}(t)^{\mathfrak{r}+\mathfrak{R}(\frac{1}{2})}\right)\mathfrak{Z}(t)^\frac{1}{2},
  \end{align*}
  \begin{align*}
    |K_2|\leq&C(T)\|\theta\|_{L_{t,x}^\infty}^2\int_{0}^{t}\left\|\frac{u_x(s)}{\sqrt{v(s)}}\right\|^2ds\\
    \leq&C(T)\left(1+\mathfrak{Y}(t)^{2\mathfrak{r}}\right)\left(1+\mathfrak{Y}(t)^{\mathfrak{R}(1)}\right)\\
    \leq&C(T)\left(1+\mathfrak{Y}(t)^{2\mathfrak{r}+\mathfrak{R}(1)}\right),
  \end{align*}
  \begin{align*}
    |K_3|\leq&C(T)\|\theta\|_{L_{t,x}^\infty}^{1+\frac{\beta}{2}}\left(\int_{0}^{t} \left\|\frac{u_x(s)}{\sqrt{v(s)}}\right\|^2ds\right)^\frac{1}{2} \left(\int_{0}^{t}\left\|\frac{\sqrt{\kappa(v(s),\theta(s))}\theta_t(s)}{\sqrt{v(s)}}\right\|^2ds\right)^\frac{1}{2}\\
    \leq&C(T)\left(1+\mathfrak{Y}(t)^{\frac{\beta+2}{2}\mathfrak{r}}\right)\left(1+\mathfrak{Y}(t)^{\mathfrak{R}(\frac{1}{2})}\right)\mathfrak{Z}(t)^\frac{1}{2}\\
    \leq&C(T)\left(1+\mathfrak{Y}(t)^{\frac{\beta+2}{2}\mathfrak{r}+\mathfrak{R}(\frac{1}{2})}\right)\mathfrak{Z}(t)^\frac{1}{2},
  \end{align*}
  \begin{align*}
    |K_4|\leq&C(T)\|\theta(t)\|_{L^\infty}^{\beta+1}\|u_x(t)\|^2\\
    \leq&C(T)\left(1+\mathfrak{Y}(t)^{(\beta+1)\mathfrak{r}}\right)\left(1+\mathfrak{Y}(t)^{\mathfrak{R}(1)}\right)\mathfrak{Z}(t)^\frac{1}{2}\\
    \leq&C(T)\left(1+\mathfrak{Y}(t)^{(\beta+1)\mathfrak{r}+\mathfrak{R}(1)}\right)\mathfrak{Z}(t)^\frac{1}{2},
  \end{align*}
  \begin{align*}
    |K_5|\leq&C(T)\|\theta\|_{L_{t,x}^\infty}^{\beta+1}\left(\int_{0}^{t} \left\|\frac{u_x(s)}{\sqrt{v(s)}}\right\|ds\right)^\frac{1}{2} \left(\int_{0}^{t}\left\|\frac{u_{xt}(s)}{\sqrt{v(s)}}\right\|ds\right)^\frac{1}{2}\\
    \leq&C(T)\left(1+\mathfrak{Y}(t)^{(\beta+1)\mathfrak{r}}\right)\left(1+\mathfrak{Y}(t)^{\mathfrak{R}(1)}\right)\\
    &\times\left(1+\mathfrak{X}(t)^\frac{1}{2}+\mathfrak{Y}(t)^\frac{1}{2}+\mathfrak{Y}(t)^{2\mathfrak{r}+\mathfrak{R}(1)}+\mathfrak{Z}(t)^{\frac{\mathfrak{R}_2}{2}}\right)\\
    \leq&C(T)\Big(1+\mathfrak{Y}(t)^{(\beta+1)\mathfrak{r}+\mathfrak{R}(1)}\mathfrak{X}(t)^\frac{1}{2}+\mathfrak{Y}(t)^{(\beta+1)\mathfrak{r}+\mathfrak{R}(1)+\frac{1}{2}}\\
    &+\mathfrak{Y}(t)^{(\beta+3)\mathfrak{r}+\mathfrak{R}(2)}+\mathfrak{Y}(t)^{(\beta+1)\mathfrak{r}+\mathfrak{R}(1)}\mathfrak{Z}(t)^{\frac{\mathfrak{R}_2}{2}}\Big),
  \end{align*}
  \begin{align*}
    |K_6|\leq&C(T)\|\theta\|_{L_{t,x}^\infty}^{\beta+1}\|u_x\|_{L_{t,x}^\infty} \int_{0}^{t}\left\|\frac{u_x(s)}{\sqrt{v(s)}}\right\|^2ds\\
    \leq&C(T)\left(1+\mathfrak{Y}(t)^{(\beta+1)\mathfrak{r}}\right)\left(1+\mathfrak{Y}(t)^{\mathfrak{R}(\frac{3}{4})}\right)\mathfrak{Z}(t)^\frac{3}{8}\left(1+\mathfrak{Y}(t)^{\mathfrak{R}(1)}\right)\\
    \leq&C(T)\left(1+\mathfrak{Y}(t)^{(\beta+1)\mathfrak{r}+\mathfrak{R}(\frac{7}{4})}\right)\mathfrak{Z}(t)^\frac{3}{8},
  \end{align*}
\begin{align*}
    |K_7|\leq&C(T)\|\theta\|_{L_{t,x}^\infty}\left(\int_{0}^{t}\left\|\frac{\theta_t(s)}{\sqrt{v(s)}}\right\|^2ds\right)^\frac{1}{2} \left(\int_{0}^{t}\left\|\frac{u_x(s)}{\sqrt{v(s)}}\right\|^2ds\right)^\frac{1}{2}\\
    \leq&C(T)\left(1+\mathfrak{Y}(t)^{\mathfrak{r}}\right)\mathfrak{X}(t)^\frac{1}{2}\left(1+\mathfrak{Y}(t)^{\mathfrak{R}(\frac{1}{2})}\right)\\
    \leq&C(T)\left(1+\mathfrak{Y}(t)^{\mathfrak{r}+\mathfrak{R}(\frac{1}{2})}\right)\mathfrak{X}(t)^\frac{1}{2},
    \end{align*}
  \begin{align*}
    |K_8|\leq&C(T)\|\theta\|_{L_{t,x}^\infty}^2\int_{0}^{t}\left\|\frac{u_x(s)}{\sqrt{v(s)}}\right\|^2ds\\
    \leq&C(T)\left(1+\mathfrak{Y}(t)^{2\mathfrak{r}}\right)\left(1+\mathfrak{Y}(t)^{\mathfrak{R}(1)}\right)\\
    \leq&C(T)\left(1+\mathfrak{Y}(t)^{2\mathfrak{r}+\mathfrak{R}(1)}\right),
  \end{align*}
  \begin{align*}
    |K_9|\leq&C(T)\|\theta\|_{L_{t,x}^\infty}\left(\int_{0}^{t} \left\|\frac{u_x(s)}{\sqrt{v(s)}}\right\|^2ds\right)^\frac{1}{2} \left(\int_{0}^{t}\left\|\frac{q_x(s)}{\sqrt{v(s)}}\right\|^2ds\right)^\frac{1}{2}\\
    \leq&C(T)\left(1+\mathfrak{Y}(t)^{\mathfrak{r}}\right)\left(1+\mathfrak{Y}(t)^{\mathfrak{R}(\frac{1}{2})}\right)\left(1+\mathfrak{Y}(t)^{\mathfrak{R}(4)}\right)\\
    \leq&C(T)\left(1+\mathfrak{Y}(t)^{\mathfrak{r}+\mathfrak{R}(\frac{9}{2})}\right),
  \end{align*}
  \begin{align*}
    |K_{10}|\leq&C(T)\|\theta\|_{L_{t,x}^\infty}\|u_x\|_{L_{t,x}^\infty} \int_{0}^{t}\left\|\frac{u_x(s)}{\sqrt{v(s)}}\right\|^2ds\\
    \leq&C(T)\left(1+\mathfrak{Y}(t)^{\mathfrak{r}}\right)\left(1+\mathfrak{Y}(t)^{\mathfrak{R}(\frac{3}{4})}\right)\mathfrak{Z}(t)^\frac{3}{8}\left(1+\mathfrak{Y}(t)^{\mathfrak{R}(1)}\right)\\
    \leq&C(T)\left(1+\mathfrak{Y}(t)^{\mathfrak{r}+\mathfrak{R}(\frac{7}{4})}\right)\mathfrak{Z}(t)^\frac{3}{8},
  \end{align*}
  \begin{align*}
    |K_{11}|\leq&C(T)\|u_x\|_{L_{t,x}^\infty}\int_{0}^{t} \left\|\frac{\sqrt{\kappa(v(s),\theta(s))}\theta_x(s)}{\sqrt{v(s)}}\right\|^2ds\\
    \leq&C(T)\left(1+\mathfrak{Y}(t)^{\mathfrak{R}(\frac{3}{4})}\right)\mathfrak{Z}(t)^\frac{3}{8}\left(1+\mathfrak{Y}(t)^{\mathfrak{R}_1}\right)\\
    \leq&C(T)\left(1+\mathfrak{Y}(t)^{\mathfrak{R}(\frac{3}{4})+\mathfrak{R}_1}\right)\mathfrak{Z}(t)^\frac{3}{8},
  \end{align*}
  \begin{align*}
    |K_{12}|\leq&C(T)\|v\|_{L_{t,x}^\infty}^\frac{1}{4}\|\theta\|_{L_{t,x}^\infty}^\frac{\beta}{4}\|v_x\|_{L_t^\infty}\left(\int_{0}^{t}\left\|\frac{\theta_t(s)}{\sqrt{v(s)}}\right\|^2ds\right)^\frac{1}{2}\\
    &\times\left(\int_{0}^{t}\left\|\frac{\kappa(v(s),\theta(s))\theta_x(s)}{v(s)}\right\|^2ds\right)^\frac{1}{4} \left(\int_{0}^{t}\left\|\left(\frac{\kappa(v(s),\theta(s))\theta_x(s)}{v(s)}\right)_x\right\|^2ds\right)^\frac{1}{4}\\
    \leq&C(T)\|v\|_{L_{t,x}^\infty}^\frac{1}{4}\|\theta\|_{L_{t,x}^\infty}^\frac{\beta}{4}\|v_x\|_{L_t^\infty}\left(\int_{0}^{t}\left\|\frac{\theta_t(s)}{\sqrt{v(s)}}\right\|^2ds\right)^\frac{1}{2} \left(\int_{0}^{t}\left\|\frac{\sqrt{\kappa(v(s),\theta(s))}\theta_x(s)}{v(s)}\right\|^2ds\right)^\frac{1}{4}\\
    &\times\bigg[\|v\|_{L_{t,x}^\infty}^\frac{1}{4}\left(\int_{0}^{t} \left\|\frac{\theta_t(s)}{\sqrt{v(s)}}\right\|^2ds\right)^\frac{1}{4} +\|\theta\|_{L_{t,x}^\infty}^\frac{1}{2}\left(\int_{0}^{t}\left\| \frac{u_x(s)}{\sqrt{v(s)}}\right\|^2ds\right)^\frac{1}{4}\\
    &+\|v\|_{L_{t,x}^\infty}^\frac{1}{2}\left(\int_{0}^{t}\left\|\frac{q_x(s)}{\sqrt{v(s)}}\right\|^2ds\right)^\frac{1}{4} +\|u_x\|_{L_{t,x}^\infty}^\frac{1}{2}\left(\int_{0}^{t}\left\|\frac{u_x(s)}{\sqrt{v(s)}}\right\|^2ds\right)^\frac{1}{4}\bigg]\\
    \leq&C(T)\left(1+\mathfrak{Y}(t)^{\mathfrak{R}(\frac{1}{4})}\right)\left(1+\mathfrak{Y}(t)^{\frac{\beta}{4}\mathfrak{r}}\right)\left(1+\mathfrak{Y}(t)^{\frac{\mathfrak{R}_1}{2}}\right)\mathfrak{X}(t)^\frac{1}{2}\left(1+\mathfrak{Y}(t)^{\frac{\mathfrak{R}_1}{4}}\right)\\
    &\times\bigg[\left(1+\mathfrak{Y}(t)^{\mathfrak{R}(\frac{1}{4})}\right)\mathfrak{X}(t)^\frac{1}{4}+\left(1+\mathfrak{Y}(t)^{\frac{1}{2}\mathfrak{r}}\right)\left(1+\mathfrak{Y}(t)^{\mathfrak{R}(\frac{1}{4})}\right)\\
    &+\left(1+\mathfrak{Y}(t)^{\mathfrak{R}(\frac{1}{2})}\right)\left(1+\mathfrak{Y}(t)^{\mathfrak{R}(2)}\right)+\left(1+\mathfrak{Y}(t)^{\mathfrak{R}(\frac{3}{8})}\right)\mathfrak{Z}(t)^\frac{3}{16}\left(1+\mathfrak{Y}(t)^{\mathfrak{R}(\frac{1}{4})}\right)\bigg]\\
    \leq&C(T)\bigg(\mathfrak{X}(t)^\frac{3}{4}+\mathfrak{Y}(t)^{\mathfrak{R}(\frac{1}{2})+\frac{\beta}{4}\mathfrak{r}+\frac{3\mathfrak{R}_1}{4}}\mathfrak{X}(t)^\frac{3}{4}\\
    &+\mathfrak{Y}(t)^{\mathfrak{R}(\frac{1}{2})+\frac{\beta+2}{4}\mathfrak{r}+\frac{3\mathfrak{R}_1}{4}}\mathfrak{X}(t)^\frac{1}{2}+\mathfrak{Y}(t)^{\mathfrak{R}(\frac{11}{4})+\frac{\beta}{4}\mathfrak{r}+\frac{3\mathfrak{R}_1}{4}}\mathfrak{X}(t)^\frac{1}{2}\\
    &+\mathfrak{Y}(t)^{\mathfrak{R}(\frac{7}{8})+\frac{\beta}{4}\mathfrak{r}+\frac{3\mathfrak{R}_1}{4}}\mathfrak{X}(t)^\frac{1}{2}\mathfrak{Z}(t)^\frac{3}{16}\bigg),
  \end{align*}
  \begin{align*}
    |K_{13}|\leq&C(T)\|\theta\|_{L_{t,x}^\infty}\left(\int_{0}^{t}\left\|\frac{q_x(s)}{\sqrt{v(s)}}\right\|^2ds\right)^\frac{1}{2} \left(\int_{0}^{t}\left\|\frac{u_x(s)}{\sqrt{v(s)}}\right\|^2ds\right)^\frac{1}{2}\\
    \leq&C(T)\left(1+\mathfrak{Y}(t)^{\mathfrak{r}}\right)\left(1+\mathfrak{Y}(t)^{\mathfrak{R}(4)}\right)\left(1+\mathfrak{Y}(t)^{\mathfrak{R}(\frac{1}{2})}\right)\\
    \leq&C(T)\left(1+\mathfrak{Y}(t)^{\mathfrak{r}+\mathfrak{R}(\frac{9}{2})}\right),
  \end{align*}
and
  \begin{align*}
    |K_{14}|\leq&C(T)\|v\|_{L_{t,x}^\infty}^\frac{1}{2}\|\theta\|_{L_{t,x}^\infty}^\frac{\beta}{2} \left(\int_{0}^{t}\left\|\frac{q_x(s)}{\sqrt{v(s)}}\right\|^2ds\right)^\frac{1}{2} \left(\int_{0}^{t}\left\|\frac{\sqrt{\kappa(v(s),\theta(s))\theta_t(s)}}{\sqrt{v(s)}}\right\|^2ds\right)^\frac{1}{2}\\
    \leq&C(T)\left(1+\mathfrak{Y}(t)^{\mathfrak{R}(\frac{1}{2})}\right)\left(1+\mathfrak{Y}(t)^{\frac{\beta}{2}\mathfrak{r}}\right)\left(1+\mathfrak{Y}(t)^{\mathfrak{R}(4)}\right)\mathfrak{Z}(t)^\frac{1}{2}\\
    \leq&C(T)\left(1+\mathfrak{Y}(t)^{\frac{\beta}{2}\mathfrak{r}+\mathfrak{R}(\frac{9}{2})}\right)\mathfrak{Z}(t)^\frac{1}{2}.
  \end{align*}

Putting the above estimates on $K_i (i=1,2,\cdots, 14)$ into \eqref{estimate-X+Y}, we can deduce \eqref{estimate-X(t)+Y(t)}. This completes the proof of Lemma \ref{estimate for X Y}.
\end{proof}

With the above estimates in hand, we can deduce the desired estimates on the upper bounds of the specific volume $v(t,x)$ and the absolute temperature $\theta(t,x)$.
\begin{lemma}\label{basic energy estimate}
Let $\beta>13$ and assume that the conditions listed in Lemma \ref{a priori estimate} are satisfied, there exists a positive constant $C(T)$ which depends only on the initial data $(v_0(x), u_0(x), \theta_0(x))$ and $T$ such that
  \begin{equation}\label{H1-estimate}
    \|(v(t), u(t), \theta(t), q(t))\|_{1}^2+\left\|u_{xx}(t)\right\|^2+\|u_t(t)\|^2+\int_{0}^{t}\|u_{xt}(s)\|^2ds\leq C(T),\quad 0\leq t\leq T.
  \end{equation}
Furthermore,
  \begin{equation}\label{Upper bound-v-theta}
    v(t,x)\leq C(T),\quad\theta(t,x)\leq C(T)
  \end{equation}
holds for all $(t,x)\in[0,T]\times\mathbb{T}$.
\end{lemma}
\begin{proof}
Since $\beta>13$, we have by Lemma \ref{appendix} that $0<\mathfrak{R}_i<1(i=2,\cdots,5)$,  $0<2\mathfrak{r}+\mathfrak{R}(1)<1$, $0<2\mathfrak{r}+\mathfrak{R}_1<1$. Such a fact together with Lemma \ref{estimate for Z}, Lemma \ref{estimate for X Y} and Young's inequality imply that
  \begin{equation*}
    \mathfrak{Z}(t)\leq C(T),
  \end{equation*}
and consequently
  \begin{equation*}
    \mathfrak{X}(t)+\mathfrak{Y}(t)\leq C(T).
  \end{equation*}
This completes the proof of Lemma \ref{basic energy estimate}.
\end{proof}

To prove the estimate \eqref{A priori esitmiates-on-H2-norm} stated in Lemma \ref{a priori estimate}, we only need to estimate $\int^t_0\|\theta_{xx}(s)\|^2ds$ and $\|q_{xx}(t)\|$. The next lemma is concerned with $\int^t_0\|\theta_{xx}(s)\|^2ds$.
\begin{lemma}\label{energy estimate for thetaxx}
Under the conditions listed in Lemma \ref{a priori estimate}, there exists a positive constant $C(T)$ which depends only on the initial data $(v_0(x), u_0(x), \theta_0(x))$ and $T$ such that
  \begin{equation}\label{L2L2-estimate-theta-xx}
    \int_{0}^{t}\|\theta_{xx}(s)\|^2ds\leq C(T), \quad 0\leq t\leq T.
  \end{equation}
\end{lemma}
\begin{proof}
Differentiating both sides of the equation \eqref{Radiation-NS}$_2$ for the conservation of momentum, multiplying the resultant equation by $\theta_x(t,x)$, and integrating it over $[0,t]\times\mathbb{T}$, we have
  \begin{align}\label{estimate-L2L2-theta-xx}
    &\frac{R}{\gamma-1}\int_{\mathbb{T}}\frac{1}{2}\theta_x^2dx +\int_{0}^{t}\int_{\mathbb{T}}\frac{\kappa(v,\theta)\theta_{xx}^2}{v}dxds\nonumber\\
    =&\frac{R}{\gamma-1}\int_{\mathbb{T}}\frac{1}{2}\theta_{0x}^2dx +R\underbrace{\int_{0}^{t}\int_{\mathbb{T}}\frac{\theta_{xx}\theta u}{v}dxds}_{L_1} +\underbrace{\int_{0}^{t}\int_{\mathbb{T}}\frac{\kappa(v,\theta)\theta_{xx}\theta_xv_x}{v^2}dxds}_{L_2}\\
    &+\underbrace{\int_{0}^{t}\int_{\mathbb{T}}\theta_{xx}q_xdxds}_{L_3} +\mu\underbrace{\int_{0}^{t}\int_{\mathbb{T}}\frac{\theta_{xx}u_x^2}{v}dxds}_{L_4}.\nonumber
  \end{align}
The last four terms $L_i (i=1,2,3,4)$ can be estimated as follows:
  \begin{align*}
    |L_1|\leq&\varepsilon\int_{0}^{t}\int_{\mathbb{T}}\frac{\theta_{xx}^2}{v}dxds +C_\varepsilon(T)\|\theta(t)\|\|\theta_x(t)\|\|u_x(t)\|^2\\
    \leq&\varepsilon\int_{0}^{t}\int_{\mathbb{T}}\frac{\theta_{xx}^2}{v}dxds+C_\varepsilon(T),
  \end{align*}
  \begin{align*}
    |L_2|\leq&\varepsilon\int_{0}^{t}\int_{\mathbb{T}}\frac{\kappa(v,\theta)\theta_{xx}^2}{v}dxds +C_\varepsilon(T)\|v\|_{L_{t,x}^\infty}^\frac{1}{2}\|\theta\|_{L_{t,x}^\infty}^\frac{\beta}{2} \int_{0}^{t}\|\theta_x(s)\|\|\theta_{xx}(s)\|\|v_x(s)\|^2ds\\
    \leq&\varepsilon\int_{0}^{t}\int_{\mathbb{T}}\frac{\kappa(v,\theta)\theta_{xx}^2}{v}dxds +C_\varepsilon(T)\|v\|_{L_{t,x}^\infty}^\frac{1}{2}\|\theta\|_{L_{t,x}^\infty}^\frac{\beta}{2}\|\theta_x(t)\|^2\|v_x(t)\|^4,
  \end{align*}
  \begin{align*}
    |L_3|\leq&\varepsilon\int_{0}^{t}\int_{\mathbb{T}}\frac{\theta_{xx}^2}{v}dxds +C_\varepsilon(T)\|v\|_{L_{t,x}^\infty}\|q_x(t)\|^2\\
    \leq&\varepsilon\int_{0}^{t}\int_{\mathbb{T}}\frac{\theta_{xx}^2}{v}dxds+C_\varepsilon(T),
  \end{align*}
and
  \begin{align*}
    |L_4|\leq&\varepsilon\int_{0}^{t}\int_{\mathbb{T}}\frac{\theta_{xx}^2}{v}dxds+C_\varepsilon(T)\|u_x(t)\|^3\|u_{xx}(t)\|\\
    \leq&\varepsilon\int_{0}^{t}\int_{\mathbb{T}}\frac{\theta_{xx}^2}{v}dxds+C_\varepsilon(T).
  \end{align*}

Inserting the above estimates into \eqref{estimate-L2L2-theta-xx} and by choosing $\varepsilon>0$ sufficiently small, we can deduce the estimate \eqref{estimate-L2L2-theta-xx} from the estimates obtained in Lemma \ref{basic energy estimate}.
\end{proof}

For the estimate of $\|q_{xx}(t)\|$, we have
\begin{lemma}\label{energy estimate for qxx}
Under the conditions listed in Lemma \ref{a priori estimate}, there exists a positive constant $C(T)$ which depends only on the initial data $(v_0(x), u_0(x), \theta_0(x))$ and $T$ such that
\begin{equation}\label{estimate-qxx-L2}
  \left\|q_{xx}(t)\right\|^2\leq C(T),\quad 0\leq t\leq T.
  \end{equation}
\end{lemma}
\begin{proof}
  Differentiating both sides of the equation \eqref{Radiation-NS}$_4$ satisfied by the radiation flux $q(t,x)$, multiplying the resultant equation by $q_x(t,x)$, and then integrating it with over $\mathbb{T}$, we have
  \begin{equation}\label{qxx-L2}
    \int_{\mathbb{T}}\frac{q_{xx}^2}{v}dx +\int_{\mathbb{T}}avq_x^2(dx=\underbrace{\int_{\mathbb{T}}\frac{v_xq_xq_{xx}}{v^2}dx}_{M_1} +a\underbrace{\int_{\mathbb{T}}v_xqq_xdx}_{M_2}+4b\underbrace{\int_{\mathbb{T}}q_{xx}\theta^3\theta_xdx}_{M_3}.
  \end{equation}
  Since
  \begin{align*}
    |M_1|\leq&\varepsilon\int_{\mathbb{T}}\frac{q_{xx}^2}{v}dx+C_\varepsilon(T)\|v_x(t)\|^2\|q_x(t)\|\|q_{xx}(t)\|\\
    \leq&\varepsilon\int_{\mathbb{T}}\frac{q_{xx}^2}{v}dx+C_\varepsilon(T)\|v_x(t)\|^4\|q_x(t)\|^2\\
    \leq&\varepsilon\int_{\mathbb{T}}\frac{q_{xx}^2}{v}dx+C_\varepsilon(T),
  \end{align*}
  \begin{align*}
    |M_2|\leq&\|q_x(t)\|^2+\|q(t)\|\|q_x(t)\|\|v_x(t)\|^2\\
    \leq&C(T),
  \end{align*}
and
  \begin{align*}
    |M_3|\leq&\varepsilon\int_{\mathbb{T}}\frac{q_{xx}^2}{v}dx +C_\varepsilon(T)\|\theta(t)\|_{L_{t,x}^\infty}^6\|\theta_x(t)\|^2\\
    \leq&\varepsilon\int_{\mathbb{T}}\frac{q_{xx}^2}{v}dx+C_\varepsilon(T),
  \end{align*}
we can deduce \eqref{estimate-qxx-L2} by putting the above estimate into \eqref{qxx-L2}, by choosing $\varepsilon>0$ sufficiently small, and from the estimates obtained in Lemma \ref{basic energy estimate}.
\end{proof}

From the estimates \eqref{Estimate-lower-bound-v} obtained in Lemma \ref{lower bound for specific volume}, \eqref{H1-estimate} and \eqref{Upper bound-v-theta} obtained in Lemma \ref{basic energy estimate}, \eqref{L2L2-estimate-theta-xx} established in Lemma \ref{energy estimate for thetaxx}, and \eqref{estimate-qxx-L2} derived in Lemma \ref{energy estimate for qxx}, we can deduce that to complete the proof of the {\it a priori} estimates stated in Lemma \ref{a priori estimate}, we only need to
get the positive lower bound of the absolute temperature $\theta(t,x)$ and this the main content of the next subsection.

\subsection{Estimate on the lower bound of the absolute temperature}
Now we devote ourselves to deriving the positive lower bound for the absolute temperature. Our main idea, as stated in the introduction, is to derive a pointwise estimate between $q_x(t,x)$, $v(t,x)$, and $\theta(t,x)$. To this end, we first recall the Fourier transform on one-dimensional torus $\mathbb{T}$.
\begin{definition}
For a complex-valued function $f(x)\in L^{1}\left(\mathbb{T}\right)$ and $m\in \mathbb{Z}$, we define
\begin{equation}
  \widehat{f}(m)=\int_{\mathbb{T}} f(x) e^{-2 \pi i m x} d x,
\end{equation}
where $i$ denotes the unit imaginary number. We call $\widehat{f}(m)$ the $m-$th Fourier coefficient of $f(x)$. The Fourier series of $f(x)$ at $x \in \mathbb{T}$ is the series
\begin{equation}
  \sum_{m \in \mathbb{Z}} \widehat{f}(m) e^{2 \pi i m x} .
\end{equation}
\end{definition}
For more detailed discussions and properties of Fourier transform on one-dimensional torus $\mathbb{T}$, we refer to the classical textbook \cite{Grafakos-2014} and the references cited therein.

To derive the lower bound for the absolute temperature, it is more convenient to consider the radiation hydrodynamics system with viscosity and thermal conductivity \eqref{Radiation-NS} in Euler coordinates. Denoting
\begin{equation}\label{Euler to Lagrange}
  \Psi:(t,x)\mapsto(t,y):=\left(t,\int_{\left(0,-\frac 12\right)}^{(t,x)}(udt+vdx)\right)
\end{equation}
then noticing that $\int_{\mathbb{T}}v(t,x)dx=\int_{\mathbb{T}}v_0(x)dx=1$ and by using the inverse transformation $(\Psi)^{-1}$, we can rewrite the radiation hydrodynamics system with viscosity and thermal conductivity \eqref{Radiation-NS} in Eulerian coordinates as
\begin{eqnarray}\label{radiation hydrodynamics system in Euler coordinates}
\rho_t-(\rho u)_y&=&0,\nonumber\\
(\rho u)_{t}+(\rho u+p)_{y}&=&\left(\mu u_{y}\right)_{y},\\
\left(\rho\left(e+\frac{u^{2}}{2}\right)\right)_{t}+\left(\rho u\left(e+\frac{1}{2}u^2\right)+p u\right)_{y}+q_y&=&\left(\mu uu_{y}\right)_{y}+\left(\kappa(\rho,\theta)\theta_y\right)_{y},\nonumber\\
-q_{yy}+aq+b(\theta^4)_y&=&0.\nonumber
\end{eqnarray}
Here $(t,y)\in \mathbb{R}^+\times\mathbb{T}$.

We can get the following pointwise estimate between $q_y(t,y)$ and $\theta(t,y)$.
\begin{lemma}\label{pointwise estimate for radiative flux}
Under the assumptions listed in Lemma \ref{a priori estimate}, we can deduce that
\begin{equation}\label{pointwise-qx-theta}
  \frac{\partial q(t,y)}{\partial y}\leq b|\theta(t,y)|^4,\quad (t,y)\in[0,T]\times\mathbb{T}.
\end{equation}
\end{lemma}
\begin{proof}
  Applying Fourier transform to the equation \eqref{radiation hydrodynamics system in Euler coordinates}$_4$ satisfied by the radiation flux $q(t,y)$, we can get that
  \begin{equation*}
    \hat{q}(t,k)=-\frac{2\pi bki}{4\pi^2k^2+a}\widehat{\theta^4}(t,k)
  \end{equation*}
and consequently
\begin{eqnarray}\label{Expression-qx-theta}
  \frac{\partial q(t,y)}{\partial y}-b|\theta(t,y)|^4 &=&\sum_{k=-\infty}^{\infty}\frac{-ab}{4\pi^2k^2+a}\widehat{\theta^4}(t,k)e^{2\pi iky}\\
&=&\int_{\mathbb{T}}|\theta(t,y-z)|^4\left(\sum_{k=-\infty}^{\infty}\frac{-ab}{4\pi^2k^2+a}e^{2\pi ikz}\right)dz.\nonumber
\end{eqnarray}

Setting
\begin{equation}\label{definition of K}
  K(z)\equiv\sum_{k=-\infty}^{\infty}\frac{-ab}{4\pi^2k^2+a}e^{2\pi ikz},
\end{equation}
it is easy to see that if we can prove that
$$
K(z)\leq 0
$$
holds for all $z\in\mathbb{T}$, then we can deduce that the estimate \eqref{pointwise-qx-theta} holds.

Recall that $\mathbb{T}=\left[-\frac{1}{2},\frac{1}{2}\right]$, we can get by direct computation that
\begin{equation*}
  \frac{d^2 K(z)}{dz^2}=ab\sum_{k=-\infty}^{\infty}e^{2\pi ikz}+aK(z),\quad z\in\mathbb{T}=\left[-\frac{1}{2},\frac{1}{2}\right].
\end{equation*}
Noticing also that
\begin{equation*}
  \delta(z)=\sum_{k=-\infty}^{\infty}e^{2\pi ikz}\quad z\in\mathbb{T}=\left[-\frac{1}{2},\frac{1}{2}\right]
\end{equation*}
holds for the Dirac function $\delta(z)$, we arrive at
\begin{equation}\label{differential equation for K}
  \frac{d^2 K(z)}{dz^2}=ab\delta(z)+aK(z),\quad z\in\mathbb{T}=\left[-\frac{1}{2},\frac{1}{2}\right].
\end{equation}
On the other hand, it is easy to see that
\begin{equation*}
  K'\left(\frac{1}{2}\right)=K'\left(-\frac{1}{2}\right)=2\pi iab\sum_{k=-\infty}^{\infty}\frac{(-1)^{k+1}k}{4\pi^2k^2+a}.
\end{equation*}
Since
\begin{align*}
  2\pi iab\sum_{k=-\infty}^{\infty}\frac{(-1)^{k+1}k}{4\pi^2k^2+a}=&2\pi iab\sum_{k=0}^{\infty}\frac{(-1)^{k+1}k}{4\pi^2k^2+a}+2\pi iab\sum_{k=-\infty}^{0}\frac{(-1)^{k+1}k}{4\pi^2k^2+a}\\
  =&2\pi iab\sum_{k=0}^{\infty}\frac{(-1)^{k+1}k}{4\pi^2k^2+a}-2\pi iab\sum_{k=0}^{\infty}\frac{(-1)^{-k+1}k}{4\pi^2k^2+a}\\
  =&0,
\end{align*}
where we have used the fact that each of the above series is convergent, we can deduce that $K(z)$ satisfies the following two point boundary value problem
\begin{eqnarray}\label{BP-K(z)}
  &&\frac{d^2 K(z)}{dz^2}=ab\delta(z)+aK(z),\quad z\in\mathbb{T}=\left[-\frac{1}{2},\frac{1}{2}\right],\nonumber\\
  &&K'\left(-\frac{1}{2}\right)=K'\left(\frac{1}{2}\right)=0.
\end{eqnarray}

Solving the two point boundary problem \eqref{BP-K(z)}, one gets that
\begin{equation}\label{K(z)}
  K(z)=e^{-\sqrt{a} z} \left(\frac{1}{2} \sqrt{a} b H (-z)-\frac{\sqrt{a} e^{\sqrt{a}} b}{2 \left(e^{\sqrt{a}}-1\right)}\right)+e^{\sqrt{a} z} \left(\frac{1}{2} \sqrt{a} b H (z)-\frac{\sqrt{a} e^{\sqrt{a}} b}{2 \left(e^{\sqrt{a}}-1\right)}\right),
\end{equation}
where $H(z)$ denotes the Heaviside function.

It is easy to check that $K(z)\leq0$ holds for all $z\in \mathbb{T}=\left[-\frac{1}{2},\frac{1}{2}\right]$. Thus we complete the proof of Lemma \ref{pointwise estimate for radiative flux}.
\begin{figure}[htbp]
\centering
\includegraphics[width=0.5\textwidth]{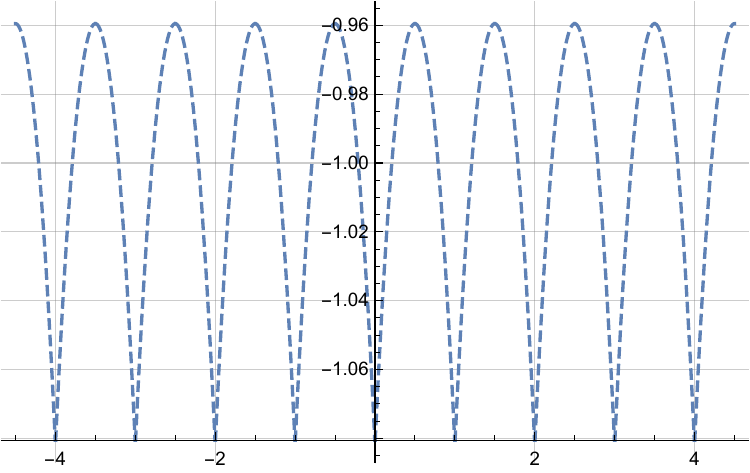}
\caption{the graph for $K(x)$ ($a=1$, $b=1$).}
\end{figure}
\end{proof}

\begin{remark}
  Using the transformation $\Psi$, we find that in the Lagrangian coordinates, $q(t,x)$ satisfies the following pointwise estimate
  \begin{equation}
  \frac{\partial q(t,x)}{\partial x}\leq bv(t,x)|\theta(t,x)|^4,\quad (t,x)\in[0,T]\times \mathbb{T}.
\end{equation}
\end{remark}

Using the above pointwise estimate, we can derive the desired positive lower bound for the absolute temperature by using the minimum principle for the second-order parabolic equations.
\begin{lemma}\label{Lower-bound-theta}
Under the assumptions listed in Lemma \ref{a priori estimate}, then for any $0\leq t\leq T$, there exists a positive constant $C(T)$ which depends only on the initial data $(v_0(x), u_0(x),\theta_0(x))$ and $T$ such that
  \begin{equation}
    \theta(t,x)\geq C(T),\quad (t,x)\in[0,T]\times\mathbb{T}.
  \end{equation}
\end{lemma}
\begin{proof}
  From the equation \eqref{Radiation-NS}$_3$ for the conservation of energy, $\theta(t,x)$ satisfies
  \begin{eqnarray*}
    \frac{R}{\gamma-1}\theta_t-\left(\frac{\kappa(v,\theta)\theta_x}{v}\right)_x&=&-q_x +\frac{\mu u_x^2}{v}-\frac{R\theta}{v}u_x\\
    &\geq& -bv\theta^4-\frac{R^2\theta^2}{4\mu v}\\
    &\geq& -C(T)\frac{R\theta}{\gamma-1}.
  \end{eqnarray*}

If we set $\overline{\theta}(t,x)=\theta(t,x) e^{C(T)t}$, then
  \begin{equation*}
    \frac{R}{\gamma-1}\overline{\theta}_t-\left(\frac{\kappa(v,\theta)\overline{\theta}_x}{v}\right)_x\geq0.
  \end{equation*}

By the minimum principle of the second-order parabolic equations, we find that
  \begin{equation*}
    \inf_{(t,x)\in\mathbb{T}\times[0,T]}\overline{\theta}(t,x)\geq\inf_{x\in\mathbb{T}}\overline{\theta}(0,x) \geq\inf_{x\in\mathbb{T}}\theta_0(x).
  \end{equation*}
  Therefore,
  \begin{equation*}
    \inf_{(t,x)\in\mathbb{T}\times[0,T]}\theta(t,x)\geq \inf_{(t,x)\in\mathbb{T}\times[0,T]}\overline{\theta}(t,x)e^{-C(T)t}\geq e^{-C(T)T}\inf_{x\in\mathbb{T}}\theta_0(x).
  \end{equation*}
The proof of Lemma \ref{Lower-bound-theta} is complete.
\end{proof}

Having obtained Lemma \ref{Lower-bound-theta}, we then finish the proof of Lemma \ref{a priori estimate}.

\section{Appendix}
In this section, we show that
\begin{lemma}\label{appendix}
For arbitrary $\beta>13$, we have $0<\mathfrak{R}_i<1(i=2,\cdots,5)$, $0<2\mathfrak{r}+\mathfrak{R}(1)<1$, $0<2\mathfrak{r}+\mathfrak{R}_1<1$.
\end{lemma}
\begin{proof} It is easy to see that
\begin{eqnarray*}
0<\mathfrak{R}_2<1&\iff&\beta>\frac{1}{4}\left(-9+\sqrt{649}\right),\\
0<\mathfrak{R}_3<1&\iff&\beta>\frac{1}{4}\left(-9+\sqrt{2697}\right),\\
0<\mathfrak{R}_4<1&\iff&\beta>13,\\
0<\mathfrak{R}_5<1&\iff&\beta>\frac{1}{12}\left(-33+\sqrt{21129}\right),\\
0<2\mathfrak{r}+\mathfrak{R}(1)<1&\iff&\beta>\frac{1}{4}\left(-7+\sqrt{89}\right),\\
0<2\mathfrak{r}+\mathfrak{R}_1<1&\iff&\beta>\frac{1}{4}\left(-7+\sqrt{537}\right).
\end{eqnarray*}
Therefore for arbitrary $\beta>13$, $0<\mathfrak{R}_i<1(i=2,\cdots,5)$, $0<2\mathfrak{r}+\mathfrak{R}(1)<1$, $0<2\mathfrak{r}+\mathfrak{R}_1<1$.
This completes the proof of Lemma \ref{appendix}.
\end{proof}

\section{Acknowledgment}
The research of Huijiang Zhao is supported in part by National Natural Science Foundation of China (Grant No. 11731008). This work is also partially supported by a grant from Science and Technology Department of Hubei Province (Grant No. 2020DFH002).

 \end{document}